\pdfoutput=1
\documentclass[11pt]{amsart}

\usepackage[utf8]{inputenc}
\usepackage[T1]{fontenc}
\usepackage{lmodern}
\usepackage{mathtools}
\usepackage{amssymb}
\usepackage{graphicx}
\usepackage{xcolor}
\usepackage{array}
\usepackage{booktabs}
\usepackage{placeins}
\usepackage[protrusion=true,expansion=false]{microtype}
\usepackage[colorlinks=true,linkcolor=blue!60!black,citecolor=blue!60!black,urlcolor=blue!60!black]{hyperref}

\title[Bitwise triangular coordinates]{%
Bitwise triangular coordinates for central products of quaternion groups\\
{\normalfont\large Floretion base vectors, digitwise \texorpdfstring{\(S_3\)}{S3}-actions, and centralizer tiles}}

\author{Creighton Dement}
\address{Independent researcher}
\email{floretionguru@gmail.com}

\hypersetup{
  pdftitle={Bitwise triangular coordinates for central products of quaternion groups: Floretion base vectors, digitwise S3-actions, and centralizer tiles},
  pdfauthor={Creighton Dement},
  pdfsubject={Quaternion groups, triangular tilings, and centralizer tiles},
  pdfkeywords={floretions, quaternions, central products, extraspecial 2-groups}
}

\subjclass[2020]{20D15; 05E18; 51M20; 52C20}
\keywords{floretions, quaternions, central products, extraspecial 2-groups, group actions, triangular tilings, centralizers}

\date{}

\numberwithin{equation}{section}

\newtheorem{thm}{Theorem}[section]
\newtheorem{prop}[thm]{Proposition}
\newtheorem{lemma}[thm]{Lemma}
\newtheorem{cor}[thm]{Corollary}

\theoremstyle{definition}
\newtheorem{dfn}[thm]{Definition}
\newtheorem{exa}[thm]{Example}

\theoremstyle{remark}
\newtheorem{rem}[thm]{Remark}

\newcommand{\Fn}{\mathcal{F}_n}
\newcommand{\Dn}{\Delta_n}
\newcommand{\An}{\mathcal{A}_n}
\newcommand{\cyc}{\gamma}
\newcommand{\Orb}{\operatorname{Orb}}
\newcommand{\Stab}{\operatorname{Stab}}
\newcommand{\Cp}{\operatorname{C}_{+}}
\newcommand{\Cn}{\operatorname{C}_{-}}
\newcommand{\Cb}{\operatorname{C}_{\mathrm{tiles}}}
\newcommand{\xnor}{\mathbin{\odot}}
\newcommand{\band}{\mathbin{\&}}

\newcommand{\li}{\overleftarrow{\imath}}
\newcommand{\ri}{\overrightarrow{\imath}}

\newcommand{\placeholderfigure}[1]{%
\fbox{\begin{minipage}[c][0.18\textheight][c]{0.24\textwidth}
\centering\scriptsize Image file not found.\\[3pt]
{\tiny\path{#1}}
\end{minipage}}}

\newcommand{\maybeincludegraphics}[2][]{%
\IfFileExists{#2}{\includegraphics[#1]{#2}}{\placeholderfigure{#2}}}

\begin{document}

\begin{abstract}
We give a self-contained treatment of the floretion coordinate model: a
bitwise and triangular coordinate model for the central product of \(n\) copies
of the quaternion group \(Q_8\).  Positive basis elements are words in
\(\{1,2,4,7\}\), identified with \(i,j,k,e\), and the real algebra they generate
is the familiar tensor product \(\An\simeq\mathbb H^{\otimes n}\).  The
contribution is the coordinate model, in which Boolean multiplication,
triangular tiling geometry, digitwise \(S_3\)-symmetry, reflection
anti-automorphisms, centralizer tiles, and parity cancellation become visible
in one language.

A local XNOR/AND rule recovers quaternionic basis multiplication in every
order.  The associated centroid map is equivariant for the digitwise
\(S_3\)-action and the dihedral action on the triangle, while synchronized
cyclic changes of noncentral digits produce equilateral centroid triangles with
an explicit symbolic product center and a parity criterion for center
coincidence.  Odd digit permutations reverse multiplication order, yielding
same-axis commutation and more general rotated axis-landing criteria.
For every non-unit basis word, the signed centralizer has cardinality \(4^n\),
and its positive tile set occupies one half of the order-\(n\) tiling and splits
according to the sign of the common product \(cb=bc\).  We show that the
positive-product component \(C_+(b)\) is invariant under the global cyclic
digit action and, apart from the identity tile, decomposes into centered
equilateral three-orbits; the negative-product component \(C_-(b)\) selects one
vertex from each of a complementary family of such orbits.  This cyclic
structure also sharpens the component counts and the vanishing identities when
\(b^2=e_n\).  Finally, a parity-dependent family shows that products of
axis-symmetric elements may land on another triangular axis, lose all
triangular reflection symmetry in odd orders, or acquire enhanced symmetry in
order two.
\end{abstract}

\maketitle

\section{Introduction}

The quaternion group \(Q_8=\{\pm i,\pm j,\pm k,\pm e\}\) admits a very compact
four-symbol basis notation.  Throughout, we use the alphabet
\[
  \{1,2,4,7\},
\]
identified with the quaternionic symbols by
\[
  1\leftrightarrow i,\qquad 2\leftrightarrow j,\qquad
  4\leftrightarrow k,\qquad 7\leftrightarrow e.
\]
Words of length \(n\) in this alphabet are multiplied coordinate by coordinate,
with the local quaternionic signs collected into one global sign.  We call these
words \emph{floretion base vectors} of order \(n\).

The abstract group obtained in this way is standard.  It is the central product
of \(n\) copies of \(Q_8\), equivalently the extraspecial \(2\)-group specified
by that central product.  The real algebra generated by the basis words is
\[
  \An \cong \mathbb H^{\otimes n},
\]
and after complexification one obtains
\[
  \An\otimes_{\mathbb R}\mathbb C \cong M_{2^n}(\mathbb C).
\]
Thus the purpose of the paper is not to introduce a new abstract algebra or a
new abstract finite group.  Rather, it studies a distinguished word basis for
\(\An\) and the structures that become visible in this basis.  The same
coordinates support a Boolean multiplication rule, a recursive triangular
tiling, digitwise \(S_3\)-actions, reflection symmetries, and centralizer tile
sets.  The payoff is a concrete finite model in which multiplication and
symmetry can be studied simultaneously.

The starting point is a two-bit encoding of the four basis symbols.  In this
encoding the unsigned part of the local product is given by bitwise XNOR, while
the sign is determined by a short Boolean expression involving XNOR and AND.
Proposition~\ref{prop:local-quaternion} identifies this rule with the usual
quaternion multiplication table.  The Boolean form gives a table-free local
multiplication rule that can be applied independently in each coordinate and
then assembled by collecting signs.  It also supports efficient computation,
since multiplication in \(\An\) can be organized by precomputed index and sign
arrays.

The second ingredient is geometric.  The same alphabet \(\{1,2,4,7\}\) labels
the four subtriangles obtained by joining the midpoints of the sides of an
equilateral triangle: three corner subtriangles labelled by \(1,2,4\), and a
central subtriangle labelled by \(7\).  Iterating this subdivision gives a
recursive triangular tiling.  The centroid map
\[
  P:\Dn\longrightarrow \mathbb R^2
\]
associates to each basis word the centroid of its tile.  A key point is that the
central digit \(7\) produces no displacement of the centroid but reverses the
orientation of all subsequent choices.

The results proved below can be summarized as follows.
\begin{enumerate}
\item The Boolean multiplication rule based on XNOR and AND agrees with
quaternionic basis multiplication in order one.  Since multiplication in order
\(n\) is digitwise, this gives the rule in every order.
\item The centroid map is equivariant with respect to the digitwise action of
\(S_3\) on the digits \(\{1,2,4\}\) and the dihedral action on the equilateral
triangle.
\item The cyclic action \(1\mapsto2\mapsto4\mapsto1\), applied globally or to
selected noncentral coordinates, produces equilateral triangles of centroids.  For a
synchronized local cycle, the cyclically ordered product gives an explicitly
signed symbolic center word, and a parity criterion determines exactly when
its centroid is the Euclidean center of the triangle.
\item Even digit permutations act by algebra automorphisms, while odd digit
permutations act by algebra anti-automorphisms.  This yields a general
axis-landing criterion for products of elements symmetric about possibly
different triangular axes; in the same-axis case the criterion reduces to
ordinary commutation.
\item For every non-unit basis word \(b\), the centralizer in the signed group
has cardinality \(4^n\), and its positive tile set occupies exactly one half of
the order-\(n\) tiling.  This centralizer tile set further splits into signed
product components \(C_+(b)\) and \(C_-(b)\).  The component \(C_+(b)\) is
invariant under the global cyclic digit action and decomposes, apart from the
identity, into centered equilateral orbits; \(C_-(b)\) is a transversal of a
complementary family of such orbits.
\item Signed centralizer components give a parity-dependent family of
axis-landing examples.  In the family analyzed below, products of elements
symmetric with respect to one triangular axis may land on a different triangular
axis or on no triangular axis.  In the order-two case, an accidental symmetry
enhancement can even make the product land on a larger symmetry class.
\end{enumerate}

Appendices are included first to prove the product formulas used in the
parity-dependent axis-landing example, and then to record some notational
changes over time and give an off-ramp to integer sequences.

Subsequent work \cite{DementEquilateral2026} develops the equilateral-completion
problem for these same triangular coordinates in greater detail, including
local and nonlocal completion, product points, and reflection symmetry.

\section{Base vectors and multiplication}

\begin{dfn}[Base vectors]
Fix \(n\geq 1\). A floretion base vector of order \(n\) is a word
\[
  b=b_1b_2\cdots b_n
\]
of length \(n\) in the alphabet \(\{1,2,4,7\}\). The set of all positive base vectors of order \(n\) is denoted by
\[
  \Dn=\{1,2,4,7\}^n.
\]
The identity element is the word
\[
  e_n:=77\cdots 7.
\]
\end{dfn}

\begin{dfn}[Signed group and real algebra]
Set
\[
  \Fn:=\{\pm b:b\in\Dn\}.
\]
This is the signed group of base vectors of order \(n\). We also write
\[
  \An:=\left\{\sum_{b\in\Dn} q_b b:q_b\in\mathbb{R}\right\}
\]
for the real algebra generated by \(\Dn\), with product extended bilinearly from the multiplication of base vectors.
\end{dfn}

\begin{rem}[Relation with central products and tensor products]
For \(n=1\) the signed basis group is \(Q_8\). More generally, \(\Fn\) is the
central product of \(n\) copies of \(Q_8\). Explicitly, it is obtained from
\(Q_8^n\) by quotienting the central subgroup
\[
  K_n=\{(\varepsilon_1,\ldots,\varepsilon_n):
       \varepsilon_r\in\{\pm 1\},\ \varepsilon_1\cdots\varepsilon_n=1\}.
\]
Thus
\[
  \Fn\simeq Q_8^n/K_n,\qquad
  |\Fn|=\frac{8^n}{2^{n-1}}=2\cdot 4^n.
\]
The real algebra generated by the basis words is the tensor product
\[
  \An\cong \mathbb H^{\otimes n},
\]
where a word \(b_1\cdots b_n\) corresponds to the tensor of the associated
quaternionic basis symbols.  Hence
\[
  \An\otimes_{\mathbb R}\mathbb C\cong M_{2^n}(\mathbb C).
\]
These identifications place the construction in standard finite group and
associative algebra theory.  The digitwise description is retained because it is
the one that interacts directly with the Boolean multiplication rule and the
triangular tiling.
\end{rem}

\subsection{The local quaternionic rule}

We use the two-bit encoding shown in Table~\ref{tab:encoding}. In octal notation this is equivalent to the three-bit representation
\[
  1=001,\qquad 2=010,\qquad 4=100,\qquad 7=111,
\]
where the XOR of the three bits is \(1\).

\begin{table}[htbp]
\centering
\begin{tabular}{cccc}
\toprule
octal digit & three-bit form & two-bit code & quaternionic symbol\\
\midrule
\(1\) & \(001\) & \((00)\) & \(i\)\\
\(2\) & \(010\) & \((01)\) & \(j\)\\
\(4\) & \(100\) & \((10)\) & \(k\)\\
\(7\) & \(111\) & \((11)\) & \(e\)\\
\bottomrule
\end{tabular}
\caption{Local encoding of the four allowed digits. The two-bit code is the one used in the XNOR/AND rule.}
\label{tab:encoding}
\end{table}

\begin{dfn}[Bitwise rule in order one]
Let
\[
  x=(ab),\qquad y=(cd),
\]
with \(a,b,c,d\in\{0,1\}\). Define
\[
  x\cdot y=(-1)^{m+1}(a\xnor c)(b\xnor d),
\]
where \(\xnor\) denotes bitwise XNOR and
\[
  m=(b\band c)+((a\xnor b)\band d)+(a\band(c\xnor d)).
\]
Here \(\band\) denotes bitwise AND.
\end{dfn}

\begin{table}[htbp]
\centering
\begin{tabular}{c|cccc}
\(\cdot\) & \(i\) & \(j\) & \(k\) & \(e\)\\
\hline
\(i\) & \(-e\) & \(k\) & \(-j\) & \(i\)\\
\(j\) & \(-k\) & \(-e\) & \(i\) & \(j\)\\
\(k\) & \(j\) & \(-i\) & \(-e\) & \(k\)\\
\(e\) & \(i\) & \(j\) & \(k\) & \(e\)
\end{tabular}
\caption{Quaternionic multiplication in the convention used throughout.}
\label{tab:qmult}
\end{table}

\begin{prop}[Local equivalence with quaternions]\label{prop:local-quaternion}
The rule of the preceding definition agrees with the quaternionic multiplication table in Table~\ref{tab:qmult}.
\end{prop}

\begin{proof}
The verification is direct on the sixteen possible products. The unsigned part of the product is
\[
  (a\xnor c)(b\xnor d),
\]
while the sign is determined by the parity of
\[
  m=(b\band c)+((a\xnor b)\band d)+(a\band(c\xnor d)).
\]
For example,
\[
  k\cdot j=(10)(01).
\]
The unsigned part is
\[
  (1\xnor 0,\ 0\xnor 1)=(0,0)=i.
\]
The three sign contributions are
\[
  0\band 0=0,\qquad (1\xnor 0)\band 1=0,\qquad
  1\band(0\xnor 1)=0.
\]
Hence \(m=0\), so the sign is \((-1)^{m+1}=-1\), and \(k\cdot j=-i\). The remaining cases give exactly Table~\ref{tab:qmult}.
\end{proof}

\begin{dfn}[Digitwise multiplication in order \(n\)]
Let
\[
  b=b_1\cdots b_n,\qquad c=c_1\cdots c_n
\]
be two base vectors in \(\Dn\). For each position \(r\), write the local quaternionic product as
\[
  b_rc_r=\varepsilon_r d_r,\qquad
  \varepsilon_r\in\{\pm 1\},\quad d_r\in\{1,2,4,7\}.
\]
We define
\[
  bc=\left(\prod_{r=1}^n\varepsilon_r\right)d_1d_2\cdots d_n.
\]
\end{dfn}

\begin{thm}[Bitwise rule in every order]
Applying the local bitwise rule in each position and collecting the local signs gives exactly the multiplication of base vectors of order \(n\).
\end{thm}

\begin{proof}
By Proposition~2.5, the bitwise rule computes each local product \(b_rc_r\) in order one, including its sign. The definition of order-\(n\) multiplication is precisely to multiply position by position and collect the product of the local signs. Thus correctness in order one implies correctness in every order.
\end{proof}

\begin{exa}
For instance,
\[
  iji\cdot jek=(i\cdot j)(j\cdot e)(i\cdot k)=k\,j\,(-j)=-kjj.
\]
\end{exa}

\section{The centroid algorithm}

The basic geometric step is a map from base vectors to points of the plane. The point associated with a word will be the centroid of the corresponding triangular tile.

\begin{dfn}[Elementary vectors]
Set
\[
\begin{aligned}
  v(1)&=(\cos 210^\circ,\sin 210^\circ),\\
  v(2)&=(\cos 90^\circ,\sin 90^\circ),\\
  v(4)&=(\cos 330^\circ,\sin 330^\circ),
\end{aligned}
\]
and
\[
  v(7)=(0,0).
\]
\end{dfn}

\begin{dfn}[Centroid map]
Fix an initial scale \(d_1>0\), and let \(b=b_1\cdots b_n\in\Dn\). Define
\[
  P(b)=\sum_{r=1}^n \sigma_r d_r v(b_r),
  \qquad
  d_r=\frac{d_1}{2^{r-1}},
\]
where \(\sigma_r\in\{\pm 1\}\) is the sign accumulated from the digits \(7\) occurring before position \(r\):
\[
  \sigma_r=(-1)^{\#\{s<r:b_s=7\}}.
\]
\end{dfn}

\begin{rem}
The digit \(7\) produces no displacement of its own, but it reverses the direction of all subsequent displacements. This reversal is essential for compatibility with the recursive triangular tiling.
\end{rem}

\begin{dfn}[Orientation of a tile]
For \(b\in\Dn\), set
\[
  N_{124}(b):=\#\{r:b_r\in\{1,2,4\}\}.
\]
The tile associated with \(b\) is oriented upward if and only if
\[
  N_{124}(b)\equiv n \pmod 2.
\]
Otherwise it is oriented downward.
\end{dfn}

\begin{rem}
The map \(P\) determines only the centroid. The orientation rule supplies the additional information needed to determine the displayed equilateral triangle.
\end{rem}

\begin{prop}[Compatibility with the recursive triangular tiling]
Let \(T_\varnothing\) be an initial equilateral triangle centered at the origin, and let \(R_0>0\) be its circumradius. At each step subdivide an equilateral triangle into the four equilateral triangles obtained by joining the midpoints of its sides: three corner triangles, labelled by \(1,2,4\), and the central triangle, labelled by \(7\). The corner directions are those of the elementary vectors \(v(1),v(2),v(4)\).

Put \(d_1=R_0/2\) and \(d_r=d_1/2^{r-1}\). If \(T_b\) is the triangle in the recursive tiling determined by the word \(b=b_1\cdots b_n\in\Dn\), then the centroid of \(T_b\) is \(P(b)\). Moreover \(T_b\) is oriented upward if and only if
\[
  N_{124}(b)\equiv n\pmod 2.
\]
\end{prop}

\begin{proof}
We argue by induction on the length of the word. The empty word corresponds to the initial triangle, whose centroid is the origin. Suppose the centroid of the triangle associated with the prefix \(b_1\cdots b_{r-1}\) is known. At level \(r-1\), the circumradius of the current triangle is \(R_0/2^{r-1}\), and the centroid of a corner subtriangle lies at distance
\[
  \frac12\,\frac{R_0}{2^{r-1}}=\frac{R_0}{2^r}=d_r
\]
from the centroid of the current triangle, in the direction of the corresponding vertex.

At step \(r\), the directions have been reversed once for each earlier occurrence of the central digit \(7\). Hence the direction sign is
\[
  \sigma_r=(-1)^{\#\{s<r:b_s=7\}}.
\]
If \(b_r\in\{1,2,4\}\), the corresponding corner choice adds the displacement \(\sigma_r d_r v(b_r)\). If \(b_r=7\), the central subtriangle is chosen: the centroid is unchanged, but the orientation is reversed for subsequent choices. Summing the contributions over \(r=1,\ldots,n\) gives the formula for \(P(b)\).

Finally, the three corner subtriangles have the same orientation as the current triangle, while the central subtriangle has the opposite orientation. Thus the final orientation is upward if and only if the number of digits \(7\) is even. Since
\[
  \#\{r:b_r=7\}=n-N_{124}(b),
\]
this condition is equivalent to \(N_{124}(b)\equiv n\pmod 2\).
\end{proof}

\begin{rem}
The factor \(\sigma_r\) is not needed merely for \(S_3\)-equivariance: the map obtained by putting all \(\sigma_r=1\) would also be \(S_3\)-equivariant. Its role is instead to make the formula compatible with the recursive triangular tiling. The digit \(7\) selects the central triangle; it produces no displacement of the centroid, but it reverses the orientation for later choices.
\end{rem}

\begin{lemma}[No cancellation of the first nonzero step]
If \(b\ne 77\cdots 7\), then \(P(b)\ne 0\).
\end{lemma}

\begin{proof}
Let \(m\) be the first position such that \(b_m\ne 7\). The step at position \(m\) has length \(d_m\). The sum of the lengths of all subsequent steps is
\[
  d_{m+1}+\cdots+d_n<d_m.
\]
Therefore the first nonzero step cannot be cancelled by the remaining tail of the path, regardless of the later directions. Hence \(P(b)\ne 0\).
\end{proof}

\section[The digitwise S3-action and centroid geometry]{The digitwise \texorpdfstring{\(S_3\)}{S3}-action and centroid geometry}

The identification between permutations of the three vertices and the dihedral symmetries of an equilateral triangle is standard; see, for example, \cite{Armstrong}.

\begin{dfn}[Digitwise action]
The group \(S_3\) acts on \(\Dn\) by permuting the digits \(\{1,2,4\}\) in each position and fixing the digit \(7\).
\end{dfn}

\begin{dfn}[Dihedral action]
Let \(D_3\) be the symmetry group of the equilateral triangle generated by the three elementary directions \(v(1),v(2),v(4)\). We identify \(D_3\) with the group of permutations of these three vectors. If \(\pi\in S_3\), let \(\rho(\pi)\in D_3\) be the unique linear symmetry of the plane such that
\[
  \rho(\pi)v(a)=v(\pi(a))\qquad\text{for }a\in\{1,2,4\}.
\]
Also set \(\rho(\pi)v(7)=v(7)=0\).
\end{dfn}

\begin{thm}[Equivariance of the centroid map]
\label{thm:centroid-equivariance}
For every \(\pi\in S_3\) and every \(b\in\Dn\),
\[
  P(\pi b)=\rho(\pi)P(b),
\]
where \(\pi b\) denotes the digitwise application of \(\pi\) to the word \(b\).
\end{thm}

\begin{proof}
Write \(b=b_1\cdots b_n\). Since \(\pi\) fixes the digit \(7\), the positions of the digits \(7\) in \(b\) and in \(\pi b\) are identical. Hence the same coefficients \(\sigma_r\) and the same lengths \(d_r\) occur in the formulas for \(P(b)\) and \(P(\pi b)\). Using the definition of \(\rho(\pi)\), we obtain
\[
\begin{aligned}
  P(\pi b)
  &=\sum_{r=1}^n \sigma_r d_r v(\pi(b_r))\\
  &=\sum_{r=1}^n \sigma_r d_r \rho(\pi)v(b_r)\\
  &=\rho(\pi)\sum_{r=1}^n \sigma_r d_r v(b_r)
   =\rho(\pi)P(b).
\end{aligned}
\]
\end{proof}

\begin{cor}[Reflections and transpositions]\label{cor:reflections-transpositions}
The digitwise transposition \((24)\) corresponds to reflection in the axis determined by \(1\); the transposition \((14)\) corresponds to reflection in the axis determined by \(2\); and the transposition \((12)\) corresponds to reflection in the axis determined by \(4\).
\end{cor}

\begin{proof}
This is the preceding theorem in the special case where \(\pi\) is a transposition.
\end{proof}

Figure~\ref{fig:i-axis-order3} illustrates the \(1\)-axis in order three: the
highlighted tiles are the words in \(\{1,7\}^3\), hence fixed by the digit
transposition \((24)\).

\begin{figure}[htbp]
\centering
\maybeincludegraphics[width=.82\textwidth]{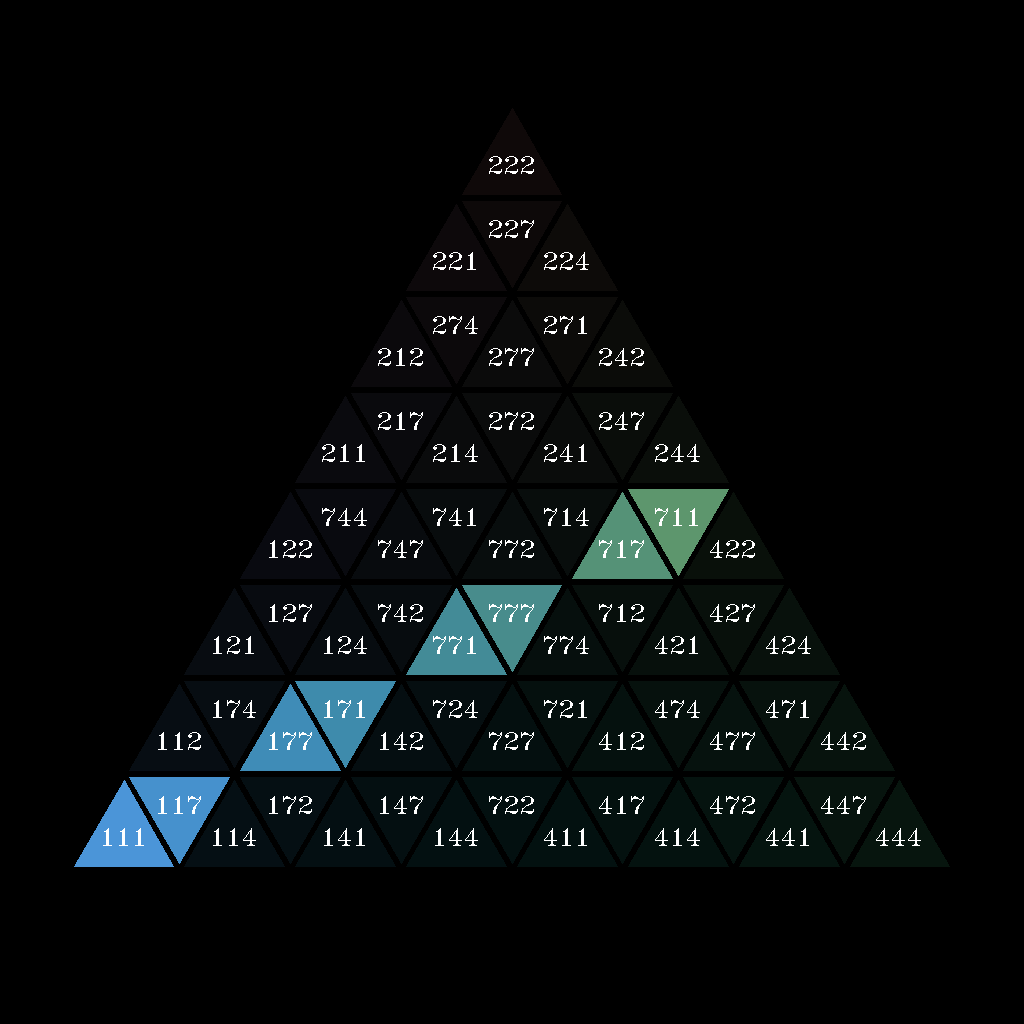}
\caption{The order-three triangular tiling with each tile labeled by its octal word.
The highlighted tiles lie on the \(1\)-axis, equivalently the \(i\)-axis; they are
the words in \(\{1,7\}^3\), fixed by the digit transposition \((24)\).}
\label{fig:i-axis-order3}
\end{figure}

\FloatBarrier

\begin{cor}[Cyclic equivariance of the centroid map]
\label{cor:cyclic-equivariance}
Let \(\cyc=(124)\), so that
\[
  1\mapsto 2,\qquad 2\mapsto 4,\qquad
  4\mapsto 1,\qquad 7\mapsto 7.
\]
Then
\[
  P(\cyc b)=R P(b),
\]
where \(R\) is the rotation of the plane sending \(v(1)\) to \(v(2)\),
\(v(2)\) to \(v(4)\), and \(v(4)\) to \(v(1)\).
\end{cor}

\begin{proof}
This is Theorem~\ref{thm:centroid-equivariance} applied to the
three-cycle \(\cyc=(124)\).
\end{proof}

\begin{cor}[Global cyclic orbits give equilateral triangles]
\label{cor:global-cyclic-orbits}
Let \(b\in\Dn\) be a base vector different from the identity \(77\cdots 7\).
Then the three points
\[
  P(b),\qquad P(\cyc b),\qquad P(\cyc^2 b)
\]
are the vertices of an equilateral triangle.
\end{cor}

\begin{proof}
By Corollary~\ref{cor:cyclic-equivariance}, the three points are
\[
  P(b),\qquad R P(b),\qquad R^2P(b),
\]
where, with the convention of Definition~3.1, \(R\) is the rotation by
\(-120^\circ\).  By the no-cancellation lemma, \(P(b)\ne 0\).  Hence the
three points are distinct.  Since they are obtained by rotating a nonzero point
about the origin through successive angles of \(120^\circ\) in magnitude, their
mutual distances are equal.
\end{proof}

\begin{dfn}[Synchronized local cyclic action]
Let \(S\subseteq\{1,\ldots,n\}\).  Define \(\cyc_S:\Dn\to\Dn\) by applying the
cycle \(1\mapsto2\mapsto4\mapsto1\) in the coordinates belonging to \(S\), while
leaving all other coordinates fixed; the digit \(7\) is fixed in every
coordinate.  For \(S=\{r\}\) we write \(\cyc_r\).  For
\(b=b_1\cdots b_n\in\Dn\), write
\[
  \operatorname{supp}(b):=\{r\in\{1,\ldots,n\}:b_r\ne7\}
\]
for its coordinate support.
\end{dfn}

For a fixed word \(b\), selected central coordinates do not affect the local
cyclic action, since
\[
  \cyc_S^t(b)
  =
  \cyc_{S\cap\operatorname{supp}(b)}^t(b),
  \qquad t=0,1,2.
\]
Thus, with no loss of generality, whenever \(b\) is fixed we assume below that
\[
  S\subseteq\operatorname{supp}(b).
\]

To make the selection visible in examples below, selected coordinates are
shown in bold blue.  For instance, if \(b=ijekiej\) and \(S=\{2,4,5\}\), we
write
\[
  b=i\textcolor{blue!60!black}{\boldsymbol{j}}e
  \textcolor{blue!60!black}{\boldsymbol{k}}
  \textcolor{blue!60!black}{\boldsymbol{i}}ej.
\]

\begin{thm}[Local cycles and their centers]
\label{thm:local-cycles-centers}
Let \(b=b_1\cdots b_n\in\Dn\), let
\[
  \varnothing\ne S\subseteq\operatorname{supp}(b),
\]
and define
\[
  b^{(t)}:=\cyc_S^t(b),\qquad t=0,1,2.
\]
Define the symbolic center word \(c=c(b,S)\) coordinatewise by
\[
  c_r=
  \begin{cases}
    7,&r\in S,\\
    b_r,&r\notin S.
  \end{cases}
\]
Then the following statements hold.
\begin{enumerate}
\item The points
\[
  P(b^{(0)}),\qquad P(b^{(1)}),\qquad P(b^{(2)})
\]
are the vertices of a nondegenerate equilateral triangle.

\item If
\[
  \nu(b):=\#\{r:b_r\ne7\},
\]
then the cyclically ordered product is
\[
  b^{(0)}b^{(1)}b^{(2)}
  =
  (-1)^{\nu(b)}c.
\]

\item The Euclidean center of the triangle is
\[
  Q=
  \sum_{r\notin S}\sigma_r(b)d_rv(b_r).
\]

\item The symbolic center is the geometric center, that is \(P(c)=Q\), if and
only if
\[
  \#\{s\in S:s<r\}\equiv0\pmod2
  \tag{\(\ast\)}\label{eq:local-center-parity}
\]
for every \(r\notin S\) such that \(b_r\ne7\).
\end{enumerate}
\end{thm}

\begin{proof}
Because \(\cyc\) fixes \(7\), the three words
\(b^{(0)},b^{(1)},b^{(2)}\) have the same \(7\)-support.  Consequently,
\[
  \sigma_r(b^{(0)})=\sigma_r(b^{(1)})
  =\sigma_r(b^{(2)})
\]
for every \(r\).  Write this common value as \(\sigma_r(b)\).

Let \(R=\rho((124))\).  With the convention of Definition~3.1, this is the
planar rotation through \(-120^\circ\), satisfying
\[
  Rv(1)=v(2),\qquad Rv(2)=v(4),\qquad Rv(4)=v(1).
\]
Split the centroid sum into coordinates outside and inside \(S\):
\[
  Q=
  \sum_{r\notin S}\sigma_r(b)d_rv(b_r),
  \qquad
  W=
  \sum_{r\in S}\sigma_r(b)d_rv(b_r).
\]
The coordinates outside \(S\) remain unchanged, while each nonzero vector in
the second sum is rotated by \(R\).  Hence
\[
  P(b^{(t)})=Q+R^tW,\qquad t=0,1,2.
\]
To see that \(W\ne0\), let \(m\) be the first coordinate in \(S\).  The
summand at \(m\) has length \(d_m\), while the sum of the lengths of all later
selected summands is strictly smaller than \(d_m\).  Thus the first nonzero
selected summand cannot be cancelled by the tail.

The three points are therefore a translation of the orbit of a nonzero vector
under a rotation of magnitude \(120^\circ\), and hence form a nondegenerate equilateral
triangle.  Moreover,
\[
  I+R+R^2=0,
\]
so their Euclidean center is
\[
  \frac13\sum_{t=0}^2P(b^{(t)})
  =
  Q+\frac13(I+R+R^2)W
  =
  Q.
\]
This proves parts \(1\) and \(3\).

For part \(2\), compute the ordered product coordinatewise.  Fix \(r\).  If
\(b_r=7\), the three local factors are \(e,e,e\), whose product is
\(e=c_r\).

Suppose that \(b_r\ne7\).  If \(r\notin S\), the three local factors are all
\(b_r\), and
\[
  b_r^3=-b_r=-c_r.
\]
If \(r\in S\), the factors occur in quaternionic cyclic order:
\[
  b_r,\qquad \cyc(b_r),\qquad \cyc^2(b_r).
\]
For every \(a\in\{i,j,k\}\),
\[
  a\,\cyc(a)\,\cyc^2(a)=-e=-c_r.
\]
Thus every noncentral coordinate of \(b\) contributes exactly one negative
sign, while every central coordinate contributes a positive sign.  Collecting
the coordinate signs gives
\[
  b^{(0)}b^{(1)}b^{(2)}
  =
  (-1)^{\nu(b)}c.
\]

It remains to prove part \(4\).  For \(r\notin S\), set
\[
  a_r:=\#\{s\in S:s<r\}.
\]
Passing from \(b\) to \(c\) changes each selected noncentral digit into \(7\).
Each such replacement introduces one additional orientation reversal for all
later coordinates.  Therefore
\[
  \sigma_r(c)=(-1)^{a_r}\sigma_r(b).
\]
Since every coordinate in \(S\) contributes zero to \(P(c)\),
\[
  P(c)
  =
  \sum_{r\notin S}
  (-1)^{a_r}\sigma_r(b)d_rv(b_r).
\]
Subtracting the expression for \(Q\) yields
\[
  P(c)-Q
  =
  -2
  \sum_{\substack{r\notin S,\ b_r\ne7\\a_r\ {\rm odd}}}
  \sigma_r(b)d_rv(b_r).
  \tag{\(\dagger\)}\label{eq:local-center-difference}
\]
If every relevant \(a_r\) is even, the sum is empty and \(P(c)=Q\).

Conversely, suppose that the sum in
\eqref{eq:local-center-difference} is nonempty, and let \(m\) be its first
index.  Its first term has length \(2d_m\), whereas the total length of all
later terms is strictly less than
\[
  2\sum_{r>m}d_r<2d_m.
\]
The first term cannot be cancelled by the tail.  Hence \(P(c)-Q\ne0\).
This proves the equivalence with \eqref{eq:local-center-parity}.
\end{proof}

See \cite{DementEquilateral2026} for a systematic study of equilateral
completion in the present centroid coordinates, where
Theorem~\ref{thm:local-cycles-centers} plays a central role.

\begin{rem}[Global cycles as local cycles]
For a fixed word \(b\), the global cyclic orbit is represented canonically by
the synchronized local cyclic action with
\[
  S=\operatorname{supp}(b),
\]
since
\[
  \cyc_{\{1,\ldots,n\}}^t(b)
  =
  \cyc_{\operatorname{supp}(b)}^t(b),
  \qquad t=0,1,2.
\]
Thus Theorem~\ref{thm:local-cycles-centers} also gives a second proof of
Corollary~\ref{cor:global-cyclic-orbits}.
\end{rem}

\begin{rem}[Orbit order]
The equilateral triangle itself is an unordered set of three vertices, but the
product formula retains the cyclic order
\[
  b,\qquad \cyc_Sb,\qquad \cyc_S^2b.
\]
Reversing the local cycle changes the local sign in each selected
noncentral coordinate.  Hence the total product sign changes exactly when
\[
  |S|
\]
is odd.
\end{rem}

\begin{cor}[Terminal selections]\label{cor:local-terminal-block}
Suppose that
\[
  S=\operatorname{supp}(b)\cap\{m,m+1,\ldots,n\}
  \ne\varnothing.
\]
Then the symbolic center \(c(b,S)\) is the Euclidean center of the local cyclic
triangle.
\end{cor}

\begin{proof}
Every noncentral coordinate outside \(S\) occurs before every element of \(S\).
Thus the number in \eqref{eq:local-center-parity} is zero for each relevant
coordinate.
\end{proof}

\begin{exa}[The same word with two local selections]
Let
\[
  b=ijekiej,
  \qquad
  \operatorname{supp}(b)=\{1,2,4,5,7\}.
\]
Two choices of \(S\) illustrate both outcomes in
Theorem~\ref{thm:local-cycles-centers}(4).

First take \(S_1=\{2,4\}\).  Its local orbit is
\[
\begin{array}{c@{\qquad}c}
  t=0 & i\textcolor{blue!60!black}{\boldsymbol{j}}e
        \textcolor{blue!60!black}{\boldsymbol{k}}iej\\
  t=1 & i\textcolor{blue!60!black}{\boldsymbol{k}}e
        \textcolor{blue!60!black}{\boldsymbol{i}}iej\\
  t=2 & i\textcolor{blue!60!black}{\boldsymbol{i}}e
        \textcolor{blue!60!black}{\boldsymbol{j}}iej
\end{array}
\qquad
\bigl(S_1=\{2,4\}\bigr).
\]
The symbolic center is
\[
  c_1=c(b,S_1)=ieeeiej.
\]
The unselected noncentral coordinates are \(r=1,5,7\), preceded by
\(0,2,2\) selected coordinates, respectively.  All three counts are even.
Hence
\[
  P(c_1)=Q_1,
\]
so the symbolic center is the Euclidean center.  Since \(\nu(b)=5\), the
ordered vertex product is \(-c_1\).

Now take \(S_2=\{2,4,5\}\).  The local orbit is
\[
\begin{array}{c@{\qquad}c}
  t=0 & i\textcolor{blue!60!black}{\boldsymbol{j}}e
        \textcolor{blue!60!black}{\boldsymbol{k}}
        \textcolor{blue!60!black}{\boldsymbol{i}}ej\\
  t=1 & i\textcolor{blue!60!black}{\boldsymbol{k}}e
        \textcolor{blue!60!black}{\boldsymbol{i}}
        \textcolor{blue!60!black}{\boldsymbol{j}}ej\\
  t=2 & i\textcolor{blue!60!black}{\boldsymbol{i}}e
        \textcolor{blue!60!black}{\boldsymbol{j}}
        \textcolor{blue!60!black}{\boldsymbol{k}}ej
\end{array}
\qquad
\bigl(S_2=\{2,4,5\}\bigr).
\]
Here
\[
  c_2=c(b,S_2)=ieeeeej.
\]
The final unselected noncentral coordinate \(r=7\) has three selected
coordinates before it.  Thus condition
\eqref{eq:local-center-parity} fails and
\[
  P(c_2)\ne Q_2.
\]
Again the ordered vertex product is \(-c_2\).  Thus the same base word can
produce either outcome: the distinction is controlled entirely by the parity
of the selected coordinates preceding each unselected noncentral coordinate.
\end{exa}

\section{Reflection anti-automorphisms and axis-landing criteria}

The geometric reflections of Corollary~\ref{cor:reflections-transpositions}
also act naturally on the algebra.  The parity of the corresponding digit
permutation determines whether multiplication order is preserved or reversed,
and this leads to the axis-landing criteria below.

\begin{prop}[Reflections reverse multiplication order]
We use the same notation for a digitwise permutation of \(\{1,2,4,7\}\) and
for its linear extension to \(\An\).  If \(\pi\in S_3\) is even, then
\[
  \pi(XY)=\pi(X)\pi(Y)\qquad\text{for all }X,Y\in\An.
\]
If \(\pi\in S_3\) is odd, then
\[
  \pi(XY)=\pi(Y)\pi(X)\qquad\text{for all }X,Y\in\An.
\]
\end{prop}

\begin{proof}
It is enough to check the statement on basis words.  In order one, the cycle
\((124)\), equivalently \(i\mapsto j\mapsto k\mapsto i\), preserves the
orientation of the quaternionic multiplication table and is an automorphism.
For example, \(ij=k\) is sent to \(jk=i\).  On the other hand, a transposition
of two of \(i,j,k\) reverses the orientation of the table and gives an
anti-automorphism.  For instance, the transposition fixing \(i\) and swapping
\(j\) and \(k\) sends \(ij=k\) to \(ik=-j\), while reversing the order gives
\(ki=j\), so the equality is restored in the form
\(\tau(xy)=\tau(y)\tau(x)\).  The remaining products are checked in the same
local table.  Applying this coordinate by coordinate and collecting the signs
gives the assertion for basis words of any order, and bilinearity gives it for
all of \(\An\).
\end{proof}

\begin{prop}[Axis-landing criterion]
For each \(a\in\{1,2,4\}\), let \(\tau_a\) denote the transposition of
\(\{1,2,4\}\) which fixes \(a\) and swaps the other two digits; geometrically,
\(\tau_a\) is the reflection in the axis determined by \(a\).  Suppose that
\(X,Y\in\An\) satisfy
\[
  \tau_p(X)=X,\qquad \tau_q(Y)=Y
\]
for some \(p,q\in\{1,2,4\}\).  Then, for any \(r\in\{1,2,4\}\), the product
\(XY\) is symmetric with respect to the \(r\)-axis, that is
\[
  \tau_r(XY)=XY,
\]
if and only if
\[
  XY=(\tau_r\tau_q)(Y)\,(\tau_r\tau_p)(X).
\]
\end{prop}

\begin{proof}
Since \(\tau_r\) is odd, the preceding proposition gives
\[
  \tau_r(XY)=\tau_r(Y)\tau_r(X).
\]
Because \(Y\) is fixed by \(\tau_q\) and \(X\) is fixed by \(\tau_p\), we have
\[
  \tau_r(Y)=\tau_r\tau_q(Y),\qquad
  \tau_r(X)=\tau_r\tau_p(X).
\]
Therefore
\[
  \tau_r(XY)=(\tau_r\tau_q)(Y)\,(\tau_r\tau_p)(X),
\]
and the criterion follows by imposing \(\tau_r(XY)=XY\).
\end{proof}

\begin{cor}[Products landing on the original axes]
Suppose that \(X,Y\in\An\) satisfy
\[
  \tau_p(X)=X,\qquad \tau_q(Y)=Y
\]
for some \(p,q\in\{1,2,4\}\), and put
\[
  \gamma_{pq}:=\tau_p\tau_q\in S_3 .
\]
Then \(XY\) is symmetric with respect to the \(p\)-axis if and only if
\[
  XY=\gamma_{pq}(Y)\,X .
\]
Similarly, \(XY\) is symmetric with respect to the \(q\)-axis if and only if
\[
  XY=Y\,\gamma_{qp}(X).
\]
\end{cor}

\begin{proof}
Apply the preceding proposition first with \(r=p\), and then with \(r=q\).
\end{proof}

\begin{cor}[Products of elements symmetric about the same axis]
Let \(\tau\) be one of the three transpositions in \(S_3\), and suppose that
\(X,Y\in\An\) are symmetric with respect to the corresponding axis, that is
\[
  \tau(X)=X,\qquad \tau(Y)=Y.
\]
Then
\[
  \tau(XY)=YX.
\]
Consequently \(XY\) is symmetric with respect to the same axis if and only if
\(XY=YX\).
\end{cor}

\begin{proof}
This is the preceding corollary in the case \(p=q\), where
\(\gamma_{pp}=\tau_p^2=\mathrm{id}\).
\end{proof}

\begin{rem}[Products landing on a different axis]
The same-axis criterion does not say that a noncommuting product has no
reflection symmetry.  It says only that the product does not remain symmetric
about the original axis.  If
\[
  \tau_a(X)=X,\qquad \tau_a(Y)=Y,
\]
then \(XY\) is symmetric with respect to the \(r\)-axis precisely when
\[
  XY=(\tau_r\tau_a)(Y)\,(\tau_r\tau_a)(X).
\]
For \(r\ne a\), the permutation \(\tau_r\tau_a\) is a nontrivial cyclic
permutation of the three noncentral digits.  Thus a product of two elements
symmetric about the same axis may land on a different reflection axis when a
cyclically rotated reversal identity holds.  In general, however, such a product
need not be symmetric about any triangular axis.
\end{rem}

\begin{cor}[Powers preserve reflection symmetry]
Let \(\tau\) be one of the three transpositions in \(S_3\).  If
\(X\in\An\) is symmetric with respect to the corresponding triangular axis,
that is,
\[
  \tau(X)=X,
\]
then
\[
  \tau(X^m)=X^m
\]
for every \(m\geq 1\).  In particular \(X^2\) is symmetric with respect to
the same axis.  Consequently, if
\[
  X^m=\sum_{b\in\Dn} a_b(m)b,
\]
then
\[
  a_b(m)=a_{\tau(b)}(m)
\]
for every \(b\in\Dn\) and every \(m\geq 1\).
\end{cor}

\begin{proof}
Since \(\tau\) is odd, it is an anti-automorphism.  Thus, for any
\(X_1,\ldots,X_m\in\An\),
\[
  \tau(X_1\cdots X_m)=\tau(X_m)\cdots\tau(X_1).
\]
Taking \(X_1=\cdots=X_m=X\) gives
\[
  \tau(X^m)=\tau(X)^m=X^m,
\]
because \(\tau(X)=X\).  Comparing coefficients of \(b\) and \(\tau(b)\) gives the
last assertion.
\end{proof}

\begin{rem}[Exponentials]
Since \(\An\) is finite-dimensional, the usual exponential series
\[
  \exp(X)=\sum_{m\geq 0}\frac{X^m}{m!}
\]
is defined for every \(X\in\An\).  The preceding corollary gives
\[
  \tau(X)=X\quad\Longrightarrow\quad \tau(\exp(X))=\exp(X).
\]
Thus reflection symmetry is preserved not only by powers, but also by the
exponential generating function of those powers.  In particular, the
one-parameter family \(\exp(tX)\) remains on the same reflection axis whenever
\(X\) does.
\end{rem}

\begin{rem}
This gives a direct link between reflection symmetry and centralizers.  In the
same-axis case, preservation of reflection symmetry under multiplication is
controlled by the ordinary centralizer.  In the different-axis case one obtains
a rotated or twisted centralizer condition, such as
\[
  XY=\gamma_{pq}(Y)X.
\]
Thus the geometry of triangular axes detects both ordinary and twisted forms of
commutation in the algebra.

The same observation also explains the behavior of inverses.  If \(X\) is
invertible, then any digit permutation or reflection sends \(X^{-1}\) to the
inverse of the transformed element.  In particular, if \(X\) is symmetric about
a triangular axis and \(X^{-1}\) exists, then \(X^{-1}\) is symmetric about the
same axis.  The symmetric units for a fixed axis are therefore closed under
inversion.  Products remain symmetric about the same axis only under the
commutation condition above, although the axis-landing criterion shows that a
noncommuting product can sometimes land on one of the other two axes.
\end{rem}

\begin{exa}[The finite Sierpi\'nski support as a symmetry test]
Let
\[
  S_E(n):=\sum_{b\in\{1,2,4\}^n} b\in\An .
\]
This is the finite corner-only support obtained by excluding the central digit
\(7\), equivalently by excluding the quaternionic unit \(e\).  It is the
order-\(n\) finite Sierpi\'nski-type support in the triangular tiling
\cite{Sierpinski,Hutchinson,Mandelbrot,Falconer}.

The element \(S_E(n)\) is fixed by every digitwise permutation in \(S_3\), and
in particular by every reflection \(\tau\).  Therefore, if \(X\in\An\) is
symmetric about the corresponding triangular axis, so that \(\tau(X)=X\), then
the anti-automorphism property gives
\[
  \tau\!\left(XS_E(n)\right)=\tau(S_E(n))\tau(X)=S_E(n)X .
\]
Consequently
\[
  XS_E(n)\ \text{is symmetric about the same axis}
  \quad\Longleftrightarrow\quad
  XS_E(n)=S_E(n)X .
\]
Thus the finite Sierpi\'nski support gives a concrete test case for the preceding
centralizer criterion: multiplying a symmetric element by \(S_E(n)\) preserves
the same reflection symmetry exactly when \(X\) centralizes \(S_E(n)\).
\end{exa}

\section{Centralizer tiles}

In this section we distinguish carefully between the centralizer in the
signed group \(\Fn\) and the corresponding set of positive tiles in \(\Dn\).
The preceding sections show why centralizers are geometrically natural: for two
elements symmetric about the same triangular axis, commutation is exactly the
condition that their product preserve that reflection symmetry.  We first count
the centralizer tile sets and identify a cyclic-orbit structure in their signed
components, then record the related parity projection and signed component
identities, and finally give a parity-dependent family in which products land
on another axis, on no triangular axis, or on a larger symmetry class.

\subsection{Centralizers and positive tile sets}

\begin{dfn}[Centralizer in the signed group]
For \(g\in\Fn\), set
\[
  C_{\Fn}(g):=\{h\in\Fn:hg=gh\}.
\]
\end{dfn}

\begin{lemma}[Conjugacy classes]
If \(b\in\Dn\) and \(b\ne e_n\), then the conjugacy class of \(b\) in \(\Fn\) is
\[
  \{b,-b\}.
\]
\end{lemma}

\begin{proof}
In each coordinate, conjugation by a quaternionic basis element sends the local symbol \(i,j,k,e\) either to itself or to its negative.  Hence every conjugate of \(b\) is necessarily \(b\) or \(-b\).

It remains to see that both occur. Since \(b\ne e_n\), there is a position \(r\) with \(b_r\in\{1,2,4\}\). Choose a base vector \(c\in\Dn\) whose entries are \(7\) in all positions except position \(r\), where we choose a quaternionic symbol that anticommutes with \(b_r\). Conjugation by \(c\) changes exactly one local sign and therefore produces the global element \(-b\).
\end{proof}

\begin{thm}[Cardinality of the centralizer]
If \(b\in\Dn\) and \(b\ne e_n\), then
\[
  |C_{\Fn}(b)|=4^n.
\]
\end{thm}

\begin{proof}
The group \(\Fn\) has order
\[
  |\Fn|=2\cdot 4^n.
\]
By the preceding lemma, the conjugacy orbit of \(b\) has cardinality \(2\). The orbit-stabilizer theorem gives
\[
  |\Fn|=|\Orb(b)|\,|\Stab(b)|.
\]
The stabilizer of the conjugation action is precisely the centralizer \(C_{\Fn}(b)\). Therefore
\[
  2\cdot 4^n=2\,|C_{\Fn}(b)|,
\]
and hence \(|C_{\Fn}(b)|=4^n\).
\end{proof}

\begin{dfn}[Positive centralizer tiles]
For \(b\in\Dn\), define the positive tile set of the centralizer by
\[
  \Cb(b):=\{c\in\Dn:cb=bc\}.
\]
We also decompose \(\Cb(b)\) according to the sign of the common product:
\[
  \Cp(b):=\{c\in\Dn:cb=bc\text{ and }cb\in\Dn\},
\]
\[
  \Cn(b):=\{c\in\Dn:cb=bc\text{ and }cb\in-\Dn\}.
\]
Thus
\[
  \Cb(b)=\Cp(b)\sqcup \Cn(b).
\]
\end{dfn}

\begin{rem}
The symbols \(+\) and \(-\) in \(\Cp\) and \(\Cn\) refer to the sign of the common product \(cb=bc\), not to the distinction between commutation and anticommutation. The set \(\Cb(b)\) is the centralizer as seen on positive tiles only.
\end{rem}

\begin{cor}[One half of the positive tiling]
If \(b\in\Dn\) and \(b\ne e_n\), then
\[
  |\Cb(b)|=\frac{4^n}{2}.
\]
In particular, the centralizer displayed as a set of positive tiles occupies exactly one half of the order-\(n\) tiles.
\end{cor}

\begin{proof}
Each \(c\in\Cb(b)\) contributes two elements to the signed centralizer, namely \(c\) and \(-c\). Conversely, every element of \(C_{\Fn}(b)\) is of the form \(c\) or \(-c\) with \(c\in\Cb(b)\). Hence
\[
  |C_{\Fn}(b)|=2|\Cb(b)|.
\]
The preceding theorem gives the result.
\end{proof}

\begin{lemma}[Unsigned part of a product]
For all \(b,c\in\Dn\), the products \(bc\) and \(cb\) have the same underlying base vector, possibly with different signs.
\end{lemma}

\begin{proof}
In order one, the unsigned part of the bitwise rule is
\[
  (a\xnor c)(b\xnor d),
\]
which is symmetric in the two factors because XNOR is commutative. Applying this in each position gives the claim in every order.
\end{proof}

\begin{cor}[Sign test]
Let \(b,c\in\Dn\). Then \(c\) centralizes \(b\) if and only if \(bc\) and \(cb\) have the same sign. If the signs are opposite, then \(bc=-cb\).
\end{cor}

\begin{proof}
By the preceding lemma, the two products have the same underlying base vector. Thus they can differ only by sign. They are equal exactly when the signs agree, and they are opposite exactly when the signs differ.
\end{proof}

\begin{thm}[Cyclic structure of the signed centralizer components]\label{thm:centralizer-cyclic-structure}
Fix \(b\in\Dn\setminus\{e_n\}\).  For \(c\in\Dn\), consider only the
coordinates \(r\in\operatorname{supp}(b)\), and let
\[
\begin{aligned}
  A&=\#\{r:c_r=b_r\},\\
  B&=\#\{r:c_r=\cyc(b_r)\},\\
  C&=\#\{r:c_r=\cyc^2(b_r)\}.
\end{aligned}
\]
Coordinates with \(c_r=e\), and all coordinates with \(b_r=e\), are omitted
from these three counts.  Then
\[
  c\in\Cp(b)
  \quad\Longleftrightarrow\quad
  A\equiv B\equiv C\pmod 2,
\]
and
\[
  c\in\Cn(b)
  \quad\Longleftrightarrow\quad
  B\equiv C\not\equiv A\pmod 2.
\]
Consequently,
\[
  \cyc(\Cp(b))=\Cp(b),
\]
and there is a disjoint decomposition
\[
  \Dn
  =\Cp(b)\sqcup\Cn(b)\sqcup\cyc(\Cn(b))\sqcup\cyc^2(\Cn(b)).
\]
Equivalently, every global \(\cyc\)-orbit is either contained in \(\Cp(b)\)
or meets \(\Cn(b)\) in exactly one element.
\end{thm}

\begin{proof}
At a coordinate where \(b_r\ne e\), the four possibilities for \(c_r\) are
\[
  e,\qquad b_r,\qquad \cyc(b_r),\qquad \cyc^2(b_r).
\]
The first possibility contributes neither a commutation mismatch nor a negative
local sign.  The equal noncentral pair \(c_r=b_r\) commutes, while the two
distinct noncentral pairs anticommute.  Hence
\[
  bc=(-1)^{B+C}cb,
\]
so \(b\) and \(c\) commute exactly when
\[
  B+C\equiv0\pmod2.
\]

For the sign of \(bc\), the three noncentral relative types satisfy
\[
  b_rb_r=-e,\qquad
  b_r\cyc(b_r)=\cyc^2(b_r),\qquad
  b_r\cyc^2(b_r)=-\cyc(b_r).
\]
Thus
\[
  \operatorname{sgn}(bc)=(-1)^{A+C}.
\]
Once \(bc=cb\), this is the sign of the common product.  Therefore
\(c\in\Cp(b)\) exactly when
\[
  B+C\equiv0,\qquad A+C\equiv0\pmod2,
\]
which is equivalent to \(A\equiv B\equiv C\pmod2\).  Likewise
\(c\in\Cn(b)\) exactly when
\[
  B+C\equiv0,\qquad A+C\equiv1\pmod2,
\]
which is equivalent to \(B\equiv C\not\equiv A\pmod2\).

Applying \(\cyc\) to \(c\), while keeping \(b\) fixed, cyclically permutes the
three relative types \(b_r,\cyc(b_r),\cyc^2(b_r)\).  Hence the condition
\(A\equiv B\equiv C\pmod2\) is invariant, proving
\(\cyc(\Cp(b))=\Cp(b)\).  Among the eight parity triples for \((A,B,C)\),
two have all three parities equal and correspond to \(\Cp(b)\).  In each of the
other six, exactly one parity differs from the other two; cyclic permutation
places that exceptional parity in the \(A\)-position exactly once.  Thus every
remaining global cyclic orbit has exactly one representative in \(\Cn(b)\),
proving the disjoint decomposition.
\end{proof}

\begin{cor}[Equal sizes of the signed components]\label{cor:centralizer-component-sizes}
If \(b\in\Dn\setminus\{e_n\}\), then
\[
  |\Cp(b)|=|\Cn(b)|=4^{n-1}.
\]
\end{cor}

\begin{proof}
Put \(x=|\Cp(b)|\) and \(y=|\Cn(b)|\).  Corollary~6.6 gives
\[
  x+y=\frac{4^n}{2},
\]
while the disjoint decomposition in
Theorem~\ref{thm:centralizer-cyclic-structure} gives
\[
  x+3y=4^n.
\]
Solving yields \(x=y=4^{n-1}\).
\end{proof}

\begin{cor}[Equilateral orbit geometry]\label{cor:centralizer-equilateral-orbits}
Let \(b\in\Dn\setminus\{e_n\}\).  Then
\[
  \Cp(b)\setminus\{e_n\}
\]
is a disjoint union of
\[
  \frac{4^{n-1}-1}{3}
\]
global \(\cyc\)-orbits.  The centroids of the three words in each orbit form a
nondegenerate equilateral triangle centered at the origin.

The set \(\Cn(b)\) is a transversal of exactly \(4^{n-1}\) further global
\(\cyc\)-orbits: each such orbit contains one element of \(\Cn(b)\), while its
other two elements anticommute with \(b\).
\end{cor}

\begin{proof}
The identity word \(e_n\) is the unique fixed point of the global cyclic action;
every other orbit has size three.  The first assertion follows from
Theorem~\ref{thm:centralizer-cyclic-structure},
Corollary~\ref{cor:centralizer-component-sizes}, and
Corollary~\ref{cor:global-cyclic-orbits}.  For \(c\in\Cn(b)\), the proof of the
theorem shows that the other two cyclic images have odd commutation-mismatch
parity with \(b\), and hence anticommute with \(b\) by Corollary~6.8.
\end{proof}

\begin{cor}[Axis words and full dihedral symmetry]\label{cor:axis-centralizer-dihedral}
Let \(\tau\) be one of the three digit transpositions, and suppose that
\(\tau(b)=b\).  Then
\[
  \tau(\Cp(b))=\Cp(b),\qquad
  \tau(\Cn(b))=\Cn(b).
\]
Thus the two signed centralizer components are symmetric about the corresponding
main axis.  Moreover, \(\Cp(b)\) is invariant under the full digitwise
\(S_3\)-action, so its centroid set is symmetric under the full dihedral group
of the triangular tiling, in particular about all three main axes.
\end{cor}

\begin{proof}
If \(b=e_n\), then \(\Cp(b)=\Dn\) and \(\Cn(b)=\varnothing\), so the assertion
is immediate.  Hence assume \(b\ne e_n\).

Let \(c\in C^\varepsilon(b)\), where \(\varepsilon\in\{+1,-1\}\), and write
\[
  cb=bc=\varepsilon d,
  \qquad d\in\Dn.
\]
Since \(\tau\) is an algebra anti-automorphism and \(\tau(b)=b\),
\[
  \tau(cb)=b\,\tau(c),\qquad
  \tau(bc)=\tau(c)b.
\]
Both are equal to \(\varepsilon\tau(d)\), so
\(\tau(c)\in C^\varepsilon(b)\).  Applying \(\tau\) again gives equality of
the sets.  Theorem~\ref{thm:centralizer-cyclic-structure} supplies invariance of
\(\Cp(b)\) under \(\cyc\); a transposition together with the three-cycle
generates \(S_3\).  The geometric statement follows from the equivariance of
the centroid map.
\end{proof}

\begin{cor}[Symmetric families of positive centralizer components]\label{cor:symmetric-centralizer-families}
Let \(R\subseteq\Dn\) be nonempty, and set
\[
  \Cp^{\cup}(R):=\bigcup_{b\in R}\Cp(b),
  \qquad
  \Cp^{\cap}(R):=\bigcap_{b\in R}\Cp(b).
\]
If \(\tau\) is a digit transposition and
\(\tau(\sum_{b\in R}b)=\sum_{b\in R}b\), then both sets are invariant under
the full digitwise \(S_3\)-action.  Hence their centroid sets have full
\(D_3\)-symmetry.
\end{cor}

\begin{proof}
For every \(\pi\in S_3\), the automorphism/anti-automorphism property gives
\[
  \pi(\Cp(b))=\Cp(\pi(b)).
\]
Linear independence of the base words turns the assumed symmetry of the sum
into \(\tau(R)=R\), so \(\tau\) preserves both the union and the intersection.
Theorem~\ref{thm:centralizer-cyclic-structure} gives
\(\cyc(\Cp(b))=\Cp(b)\) for \(b\ne e_n\), and the same is immediate for
\(b=e_n\).  Thus both sets are fixed by \(\tau\) and \(\cyc\), which generate
\(S_3\).  Equivariance of the centroid map gives the geometric conclusion.
\end{proof}

\begin{rem}[Squares and cancellations]
Let
\[
  X=\sum_{c\in\Dn} q_c c\in\An.
\]
Then
\[
  X^2=\sum_{c\in\Dn}q_c^2c^2+
       \sum_{\{c,d\}\subset\Dn,\ c\ne d}q_cq_d(cd+dc).
\]
By the sign test, for two base vectors \(c,d\), the products \(cd\) and \(dc\) have the same underlying base vector and may differ only by sign. If \(cd=-dc\), the mixed contribution \(q_cq_d(cd+dc)\) vanishes; if \(cd=dc\), it contributes \(2q_cq_dcd\). This gives a natural computational form for the square of a floretion and connects the calculation of \(X^2\) directly to the centralizer tile sets described above.
\end{rem}

\subsection{Parity projection and signed centralizer components}

We next record a parity projection identity that expresses the same
sign-cancellation principle algebraically.  The orientation of
a tile is governed by the parity of the number of noncentral digits, but the
following definition separates this parity from the visual up/down convention.

\begin{dfn}[Even and odd parts]
For a basis word \(b\in\Dn\), let
\[
  \nu(b):=N_{124}(b)\pmod 2.
\]
For \(X=\sum_{b\in\Dn}q_b b\in\An\), define
\[
  X_{\mathrm{even}}:=\sum_{\nu(b)=0}q_b b,
  \qquad
  X_{\mathrm{odd}}:=\sum_{\nu(b)=1}q_b b.
\]
Thus \(X=X_{\mathrm{even}}+X_{\mathrm{odd}}\).  We write
\(\An^{\mathrm{even}}\) and \(\An^{\mathrm{odd}}\) for the corresponding
linear subspaces.  When \(n\) is even, the even part is the upward part of the
tiling and the odd part is the downward part.  When \(n\) is odd, the two visual
orientations are interchanged.
\end{dfn}

\begin{prop}[Parity projection identities]
Let \(X=U+D\), where \(U=X_{\mathrm{even}}\) and \(D=X_{\mathrm{odd}}\).  Then:
\[
  D=X \quad\Longrightarrow\quad X^{2m}\in\An^{\mathrm{even}}
  \ \text{ and }\ X^{2m+1}\in\An^{\mathrm{odd}}
  \quad(m\geq0),
\]
and
\[
  U=X \quad\Longrightarrow\quad X^m\in\An^{\mathrm{even}}
  \quad(m\geq0).
\]
Moreover,
\[
  (X^2)_{\mathrm{odd}}=2\,(X\,X_{\mathrm{odd}})_{\mathrm{odd}}.
\]
Thus the odd part of a square is determined entirely by the interaction between
\(X\) and its odd component; the even-even and odd-odd contributions do not
survive the odd projection.
\end{prop}

\begin{proof}
Let \(\kappa:\An\to\An\) be the linear map defined on basis words by
\[
  \kappa(b)=(-1)^{N_{124}(b)}b.
\]
Equivalently, \(\kappa\) applies quaternionic conjugation in each digit.  Hence
\(\kappa\) is an anti-automorphism:
\[
  \kappa(XY)=\kappa(Y)\kappa(X).
\]
The even and odd parts are the \(+1\) and \(-1\) eigenspaces of \(\kappa\).

If \(X=D\), then \(\kappa(X)=-X\), and for powers of a single element the
reversal of order has no effect:
\[
  \kappa(X^m)=\kappa(X)^m=(-1)^mX^m.
\]
This proves the first assertion.  If \(X=U\), the same argument gives
\(\kappa(X^m)=X^m\) for all \(m\), proving the second assertion.

For the last identity, write \(X=U+D\).  Since \(\kappa(U)=U\) and
\(\kappa(D)=-D\), the anti-automorphism property gives
\[
  \kappa(X^2)=\kappa(X)\kappa(X)=(U-D)^2 .
\]
Therefore
\[
\begin{aligned}
  (X^2)_{\mathrm{odd}}
   &=\frac{X^2-\kappa(X^2)}2\\
   &=\frac{(U+D)^2-(U-D)^2}{2}\\
   &=UD+DU .
\end{aligned}
\]
Similarly,
\[
  XD=UD+D^2,
\]
while
\[
  \kappa(XD)=\kappa(D)\kappa(X)=(-D)(U-D)=-DU+D^2 .
\]
Hence
\[
  (XD)_{\mathrm{odd}}
   =\frac{XD-\kappa(XD)}2
   =\frac{UD+DU}{2}.
\]
Thus \((X^2)_{\mathrm{odd}}=2(XD)_{\mathrm{odd}}\), as claimed.
\end{proof}

\begin{rem}
For even order this last identity says simply
\[
  (X^2)_{\mathrm{down}}=2\,(X\,X_{\mathrm{down}})_{\mathrm{down}}.
\]
It is another form of the same sign-cancellation principle visible in the
preceding centralizer structure.
\end{rem}

\begin{dfn}[Sums of centralizer components]
For \(b\in\Dn\), set
\[
  \Sigma_+(b):=\sum_{c\in\Cp(b)} c,\qquad
  \Sigma_-(b):=\sum_{c\in\Cn(b)} c,\qquad
  \Sigma_{\mathrm{tiles}}(b):=\sum_{c\in\Cb(b)} c.
\]
These sums are regarded as elements of the real algebra \(\An\).
\end{dfn}

\begin{prop}[Vanishing of the two signed components]\label{prop:signed-components-vanish}
Let \(b\in\Dn\) satisfy
\[
  b^2=e_n.
\]
Then
\[
  \Sigma_-(b)\Sigma_+(b)=0
  \qquad\text{and}\qquad
  \Sigma_+(b)\Sigma_-(b)=0.
\]
\end{prop}

\begin{proof}
Write \(C^\varepsilon(b)\), with \(\varepsilon\in\{+1,-1\}\), for \(\Cp(b)\) or \(\Cn(b)\), respectively, and set
\[
  \Sigma_\varepsilon(b):=\sum_{c\in C^\varepsilon(b)}c.
\]
If \(c\in C^\varepsilon(b)\), then \(c\) centralizes \(b\) and
\[
  cb=bc=\varepsilon d
\]
for a unique \(d\in\Dn\). Since \(b^2=e_n\), we obtain
\[
  db=\varepsilon cb^2=\varepsilon c.
\]
Moreover \(d\) still centralizes \(b\), since \(d=\varepsilon cb\), and hence
\[
  bd=\varepsilon bcb=\varepsilon cbb=\varepsilon c,
\]
while \(db=\varepsilon c\) as above. Therefore \(d\in C^\varepsilon(b)\). Thus right multiplication by \(b\), after forgetting the sign, permutes \(C^\varepsilon(b)\). Consequently
\[
  \Sigma_\varepsilon(b)b=\varepsilon\Sigma_\varepsilon(b).
\]
Since all elements in \(C^\varepsilon(b)\) centralize \(b\), we also have
\[
  b\Sigma_\varepsilon(b)=\varepsilon\Sigma_\varepsilon(b).
\]

Now put
\[
  X=\Sigma_-(b)\Sigma_+(b).
\]
On one hand,
\[
  Xb=\Sigma_-(b)(\Sigma_+(b)b)=\Sigma_-(b)\Sigma_+(b)=X.
\]
On the other hand,
\[
  bX=(b\Sigma_-(b))\Sigma_+(b)=-\Sigma_-(b)\Sigma_+(b)=-X.
\]
But \(X\) centralizes \(b\), because both factors lie in the subalgebra generated by elements that centralize \(b\). Hence \(Xb=bX\), and the displayed equalities give \(X=-X\). Over \(\mathbb{R}\), this implies \(X=0\). The proof of \(\Sigma_+(b)\Sigma_-(b)=0\) is identical, with the order of the two factors reversed.
\end{proof}

\begin{cor}[Rotated vanishing]\label{cor:rotated-vanishing}
Let \(b\in\Dn\) satisfy \(b^2=e_n\).  Then, for \(r=0,1,2\),
\[
  \Sigma_+(b)\,\cyc^r(\Sigma_-(b))=0,
  \qquad
  \cyc^r(\Sigma_-(b))\,\Sigma_+(b)=0.
\]
\end{cor}

\begin{proof}
If \(b=e_n\), then \(\Sigma_-(b)=0\) and the assertion is immediate.  Otherwise,
Theorem~\ref{thm:centralizer-cyclic-structure} gives
\[
  \cyc(\Sigma_+(b))=\Sigma_+(b).
\]
The three-cycle \(\cyc\) is an algebra automorphism.  Applying \(\cyc^r\) to
the two identities of Proposition~\ref{prop:signed-components-vanish} therefore
gives the result.
\end{proof}

\begin{exa}
For \(b=ii\) in order two,
\[
  \Cn(ii)=\{ie,jk,kj,ei\},\qquad
  \Cp(ii)=\{ii,jj,kk,ee\}.
\]
Thus
\[
  \Sigma_-(ii)=ie+jk+kj+ei,\qquad
  \Sigma_+(ii)=ii+jj+kk+ee.
\]
Since \((ii)^2=e_2\), the preceding proposition gives
\[
  (ie+jk+kj+ei)(ii+jj+kk+ee)=0.
\]
\end{exa}

\begin{figure}[!htbp]
\centering
\begin{tabular}{ccc}
\maybeincludegraphics[width=.21\textwidth]{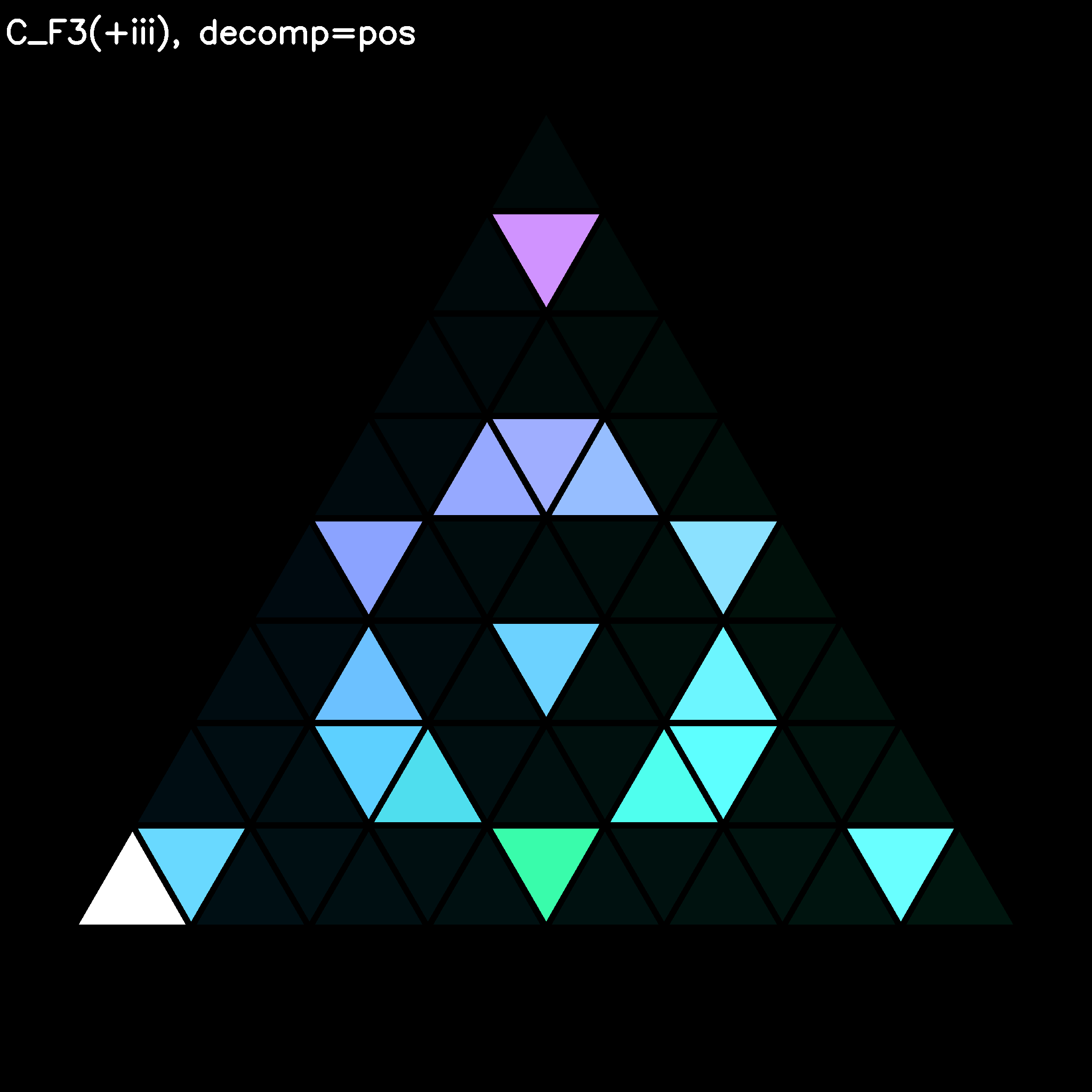} &
\maybeincludegraphics[width=.21\textwidth]{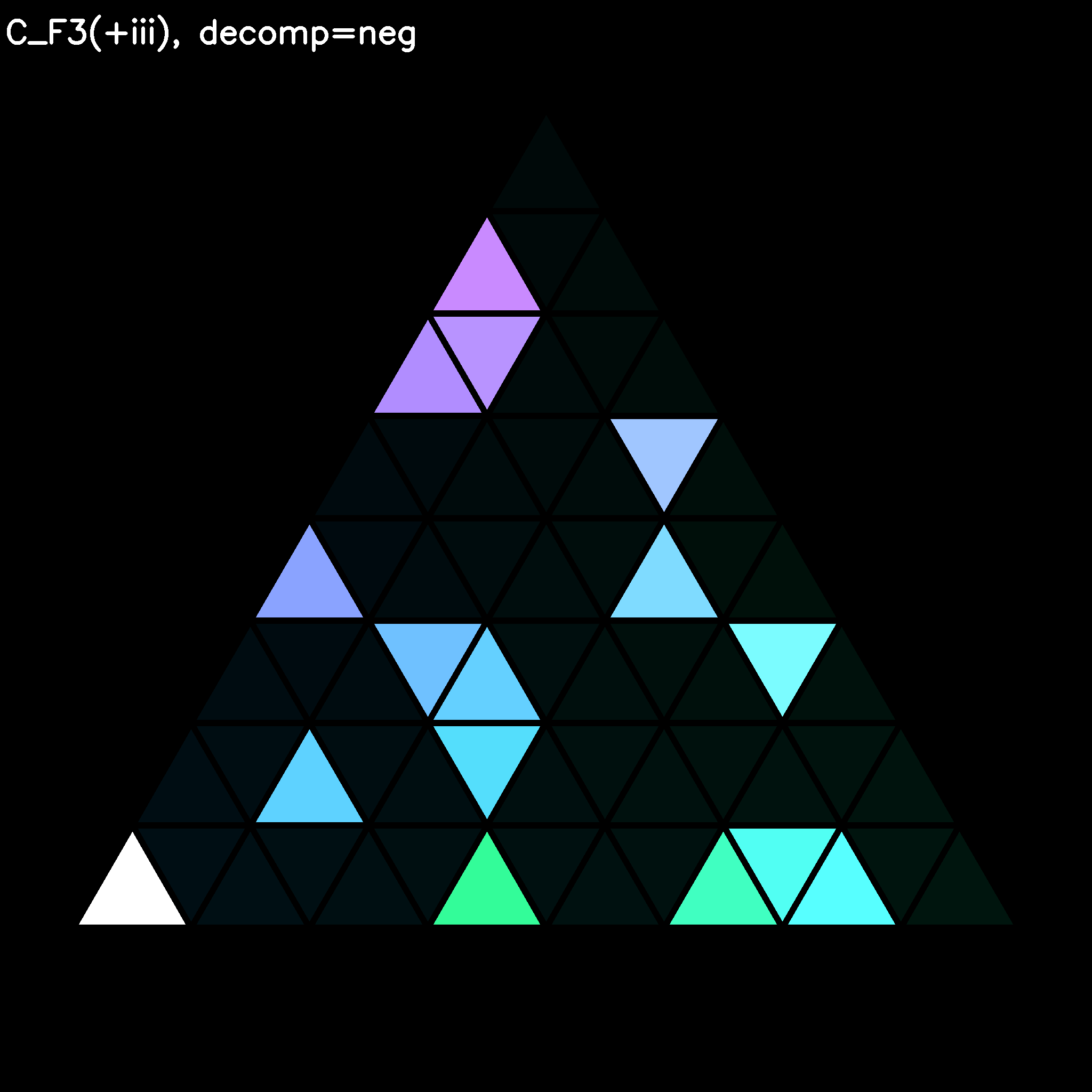} &
\maybeincludegraphics[width=.21\textwidth]{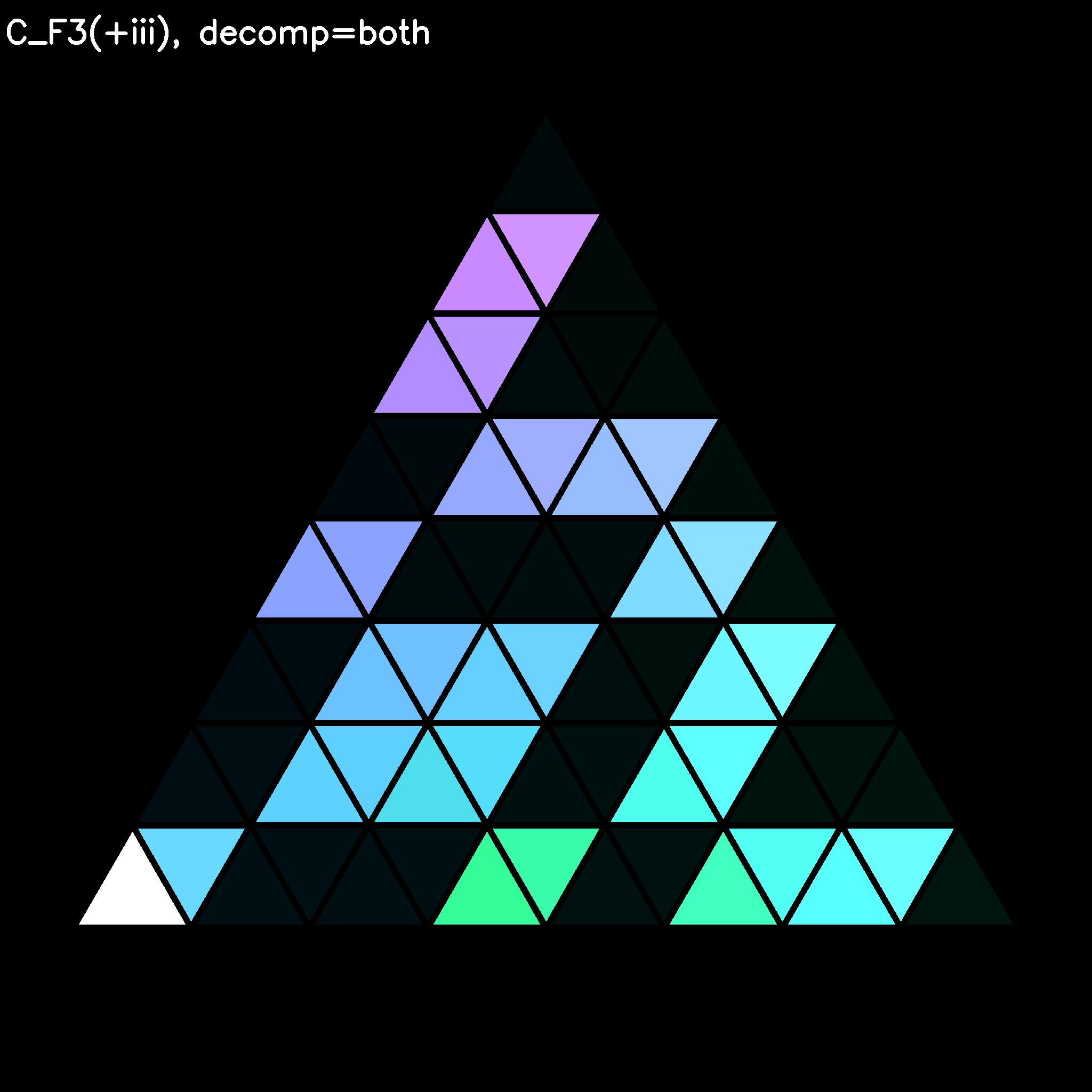}\\[-2pt]
\scriptsize (a) Positive component \(\Cp(b)\) &
\scriptsize (b) Negative component \(\Cn(b)\) &
\scriptsize (c) Union \(\Cb(b)\)
\end{tabular}
\caption{A labeled decomposition of \(\Cb(b)\) into the components \(\Cp(b)\) and \(\Cn(b)\) with base vector \(b=iii\).}
\label{fig:decomp111}
\end{figure}

\begin{figure}[!htbp]
\centering
\begin{minipage}{.45\textwidth}
\centering
\maybeincludegraphics[width=.98\linewidth,height=.24\textheight,keepaspectratio]{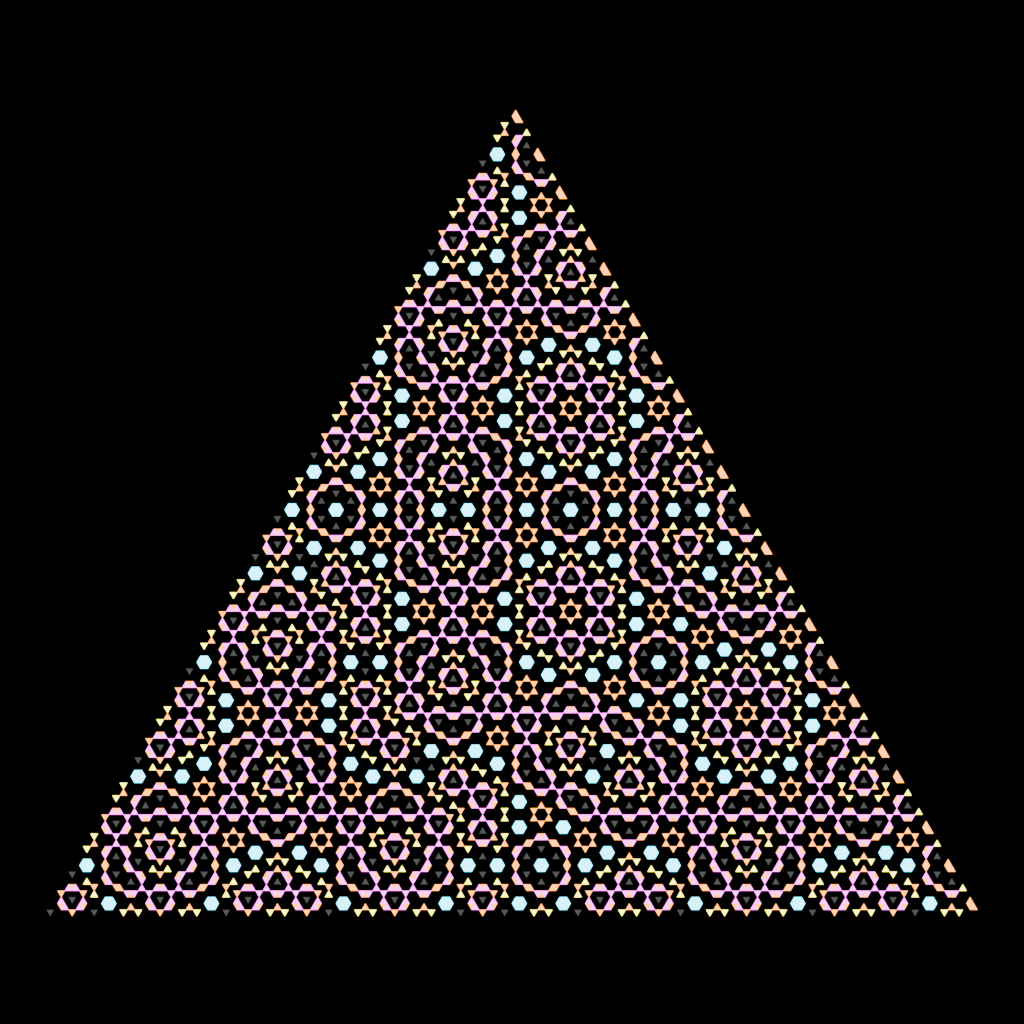}\\[-2pt]
\small (a) Negative component \(\Cn(b)\)
\end{minipage}
\hfill
\begin{minipage}{.45\textwidth}
\centering
\maybeincludegraphics[width=.98\linewidth,height=.24\textheight,keepaspectratio]{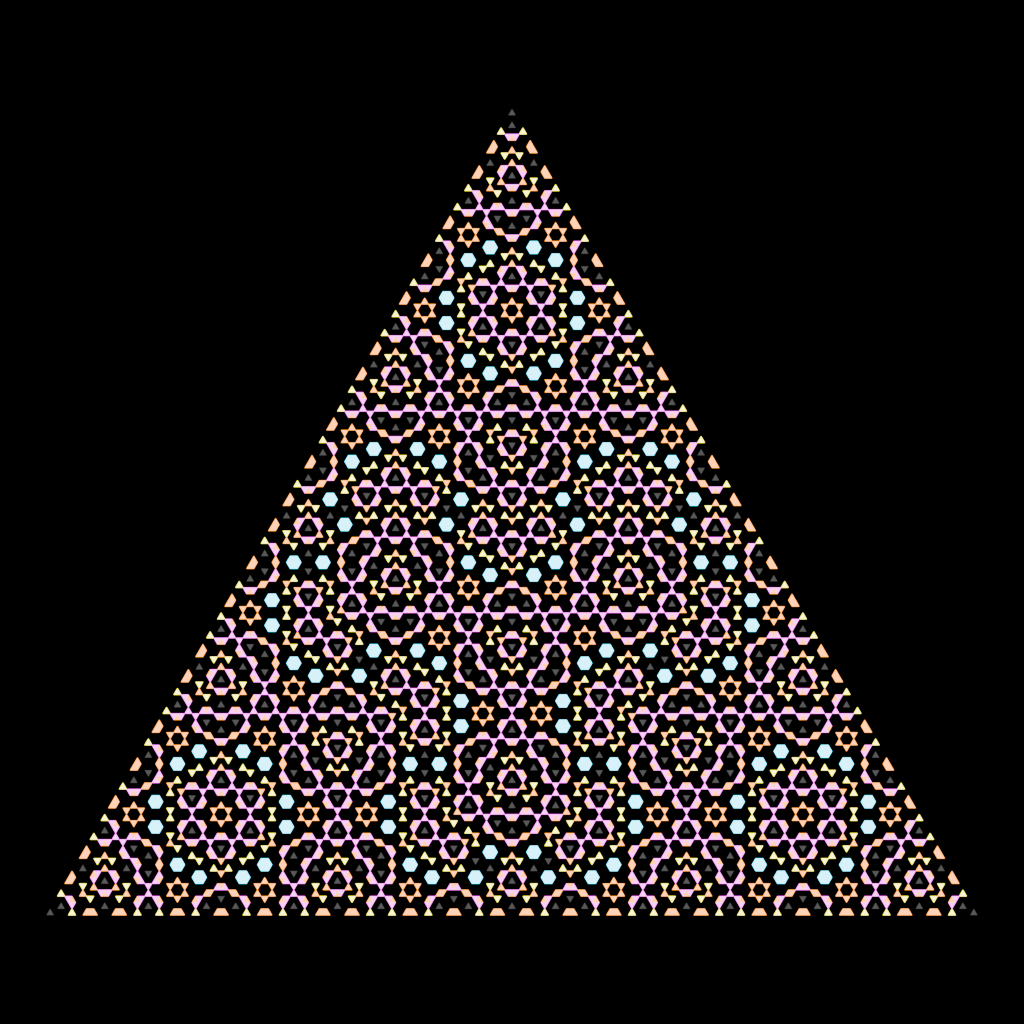}\\[-2pt]
\small (b) Positive component \(\Cp(b)\)
\end{minipage}
\caption{The two centralizer components for the base vector \(b=ieiiiii\). These pictures are not needed for the proofs, but they make the geometric structure of the centralizer tiles visible. The visualizations were generated with the Floretion add-on for Blender, with a tile's color chosen according to the number of neighboring tiles with nonzero coefficients.}
\label{fig:centralizercomponents}
\end{figure}

\FloatBarrier

Theorem~\ref{thm:centralizer-cyclic-structure} explains the threefold rotational
structure visible in the positive components of Figures~\ref{fig:decomp111}
and~\ref{fig:centralizercomponents}.  In both figures the chosen base word is
fixed by the \(I\)-axis reflection, so
Corollary~\ref{cor:axis-centralizer-dihedral} upgrades this to full dihedral
symmetry of the positive component and also explains the \(I\)-axis symmetry of
the negative component.

\subsection{Parity-dependent axis landing and accidental symmetry}
\label{subsec:parity-axis-landing}

The axis-landing criterion shows that noncommutation does not exhaust the
possibilities for symmetry.  If two elements are symmetric about the same axis
but do not commute, their product cannot remain on that original axis; however,
it may still land on another triangular axis when the corresponding rotated
reversal identity holds.  The following family gives a small test case in which
products of \(I\)-axis-symmetric elements can land on another axis, or on no
triangular axis at all.

Write
\[
  \iota_n:=\underbrace{ii\cdots i}_{n\ \mathrm{times}}\in\Dn,
  \qquad
  I_n:=\sum_{w\in\{i,e\}^n} w .
\]
Thus \(\iota_n\) is the uniform \(i\)-word of order \(n\), while \(I_n\)
is the order-\(n\) \(I\)-axis element.  Put
\[
  C_n^-:=\Sigma_-(\iota_n),\qquad
  C_n^+:=\Sigma_+(\iota_n),\qquad
  C_n^b:=\Sigma_{\mathrm{tiles}}(\iota_n).
\]
All four elements \(I_n,C_n^-,C_n^+\), and \(C_n^b\) are symmetric with
respect to the \(I\)-axis.  Nevertheless their products need not remain on
that axis.

For two symbols \(a,b\in\{e,i,j,k\}\), define
\[
  E_{ab}^{(n)}
  :=
  \sum_{\substack{w\in\{a,b\}^n\\ \#\{r:w_r=b\}\ \mathrm{even}}} w .
\]
A direct calculation gives
\begin{align}
  C_n^- I_n
    &=2^{n-1}\bigl(E_{ij}^{(n)}-E_{ek}^{(n)}\bigr), \label{eq:CminusIn}\\
  C_n^+ I_n
    &=2^{n-1}\bigl(E_{ij}^{(n)}+E_{ek}^{(n)}\bigr), \label{eq:CplusIn}\\
  C_n^b I_n
    &=2^n E_{ij}^{(n)}. \label{eq:CbIn}
\end{align}
In the reverse order one obtains
\begin{align}
  I_n C_n^-
    &=2^{n-1}\bigl(E_{ik}^{(n)}-E_{ej}^{(n)}\bigr), \label{eq:InCminus}\\
  I_n C_n^+
    &=2^{n-1}\bigl(E_{ik}^{(n)}+E_{ej}^{(n)}\bigr), \label{eq:InCplus}\\
  I_n C_n^b
    &=2^n E_{ik}^{(n)}. \label{eq:InCb}
\end{align}
The proof of these identities is given in Appendix~\ref{app:axis-landing-proof}.

The resulting axis behavior for the products \(C_n^\bullet I_n\) is:
\[
\begin{array}{c|c|c|c}
 n & C_n^-I_n & C_n^+I_n & C_n^bI_n\\
\hline
 1 & I\text{-axis} & I\text{-axis} & I\text{-axis}\\
 2 & K\text{-axis} & I,J,K\text{-axes} & K\text{-axis}\\
 \text{even }n\ge 4 & K\text{-axis} & K\text{-axis} & K\text{-axis}\\
 \text{odd }n\ge 3 & \text{none} & \text{none} & \text{none}.
\end{array}
\]
For the reverse products \(I_nC_n^\bullet\), the corresponding \(K\)-axis
entries are replaced by \(J\)-axis entries.

The parity mechanism is visible already in \(E_{ij}^{(n)}\).  By
Corollary~\ref{cor:reflections-transpositions}, the \(K\)-axis reflection is
the digit transposition that fixes \(k\) and swaps \(i\) and \(j\).  If a
word in \(\{i,j\}^n\) has \(t\) copies of \(j\), then its reflected word has
\(n-t\) copies of \(j\).  Thus the condition ``\(t\) is even'' is preserved
when \(n\) is even, but is changed to the complementary odd condition when
\(n\) is odd.  This explains why odd orders may look visually structured
while failing exact coefficientwise reflection symmetry.

The order-two case has an additional symmetry enhancement.  For instance,
\[
  C_2^+I_2=2(ii+jj+kk+ee),
\]
which is genuinely symmetric with respect to all three triangular axes.  By
contrast,
\[
  C_2^-I_2=2ii+2jj-2kk-2ee
\]
is only \(K\)-axis symmetric coefficientwise; the negative signs on the
\(kk\) and \(ee\) terms break the other two axis symmetries.

The example uses the uniform base word \(\iota_n\): every noncentral coordinate
points in the same local \(i\)-direction, so the same two parity tests are
repeated in every coordinate.  For a general base word \(b\), each noncentral
coordinate has its own local commutation and sign parity.  Thus products such
as \(\Sigma_\varepsilon(b)X\), with \(X\) chosen independently to be
axis-symmetric, form a broader parity-code problem.  They may land on another
triangular axis, on no triangular axis, or occasionally on a larger symmetry
class.

\section{Concluding remarks}

The guiding point of the paper is that the floretion basis is a distinguished
coordinate system for the familiar algebra \(\mathbb H^{\otimes n}\).  In this
basis, quaternionic multiplication, recursive triangular tilings, digitwise
\(S_3\)-actions, reflection symmetries, centralizer tile sets, parity
cancellation, and parity-dependent axis landing can be read in the same
notation.  The
results above should therefore be understood as properties of this coordinate
model rather than as claims about a new abstract group or algebra.

The equivariance theorem identifies the symbolic \(S_3\)-action with the
geometric action on the equilateral triangle.  The synchronized local-cycle
theorem then gives a direct interaction between the algebra and the geometry.
A local cyclic orbit is an equilateral triangle; its cyclically ordered product
is an explicitly signed symbolic center word, while the unchanged coordinates
determine the Euclidean center.  The parity criterion records exactly when the
centroid of the symbolic center word coincides with that Euclidean center.

Since odd digit permutations are algebra anti-automorphisms, triangular
reflection symmetry also detects commutation.  If two elements are symmetric
about the same axis, their product remains symmetric about that axis exactly
when the two elements commute.  The axis-landing criterion records the
complementary possibility: even when a product does not remain on the original
axis, it may still land on another triangular axis through a rotated reversal
identity.  Thus the triangular geometry does not merely display the basis; it
records when products preserve, reverse, or twist reflection symmetry.

The centralizer theorem gives a second manifestation of the same sign
structure.  For every non-unit basis word, the signed centralizer has
cardinality \(4^n\), and the corresponding positive tile set occupies exactly
one half of the order-\(n\) tiling.  The signed components have a sharper cyclic
geometry: \(C_+(b)\) is invariant under the global three-cycle and decomposes,
apart from the identity, into centered equilateral orbits, whereas \(C_-(b)\)
selects one vertex from each of a complementary family of such orbits.  For an
axis word, the positive component has full dihedral symmetry.  The parity
identity
\[
  (X^2)_{\mathrm{odd}}=2\,(X\,X_{\mathrm{odd}})_{\mathrm{odd}}
\]
is a compatible projection form of the same cancellation principle.  The
parity-dependent example with \(C_n^-\), \(C_n^+\), \(C_n^b\), and \(I_n\)
then shows that the signed centralizer components can behave differently under
multiplication by an external axis-symmetric element.  In particular,
coefficientwise symmetry can be broken in odd orders, redirected to a different
axis in even orders, or enhanced in small orders such as \(n=2\).

Several natural continuations arise from the interaction between the symbolic
and geometric structures developed here.  The equilateral-completion problem
for the centroids of order-\(n\) tiles is developed systematically in
\cite{DementEquilateral2026}.  That work counts and characterizes the
synchronized local-cycle family inside the full set of equilateral centroid
triangles and develops further product-point and reflection-symmetry structure.
Thus the local cyclic orbits introduced here serve both as an intrinsic feature
of the coordinate model and as a distinguished subclass in the broader
completion problem.

A second direction is to refine the sign structure underlying the centralizer
theorem.  The commutation criterion retains only the total parity of the local
sign contributions, whereas the individual contributions may carry a finer
multichannel signature.  One may ask whether the resulting signature classes
are fibers or cosets of natural maps over \(\mathbb F_2\), how their
cardinalities depend on the chosen basis word, and how they transform under
digit permutations and reflections.  Their intersections with triangular axes
and recursive subtiles may in turn provide an algebraic explanation for the
connected components, boundaries, and other geometric features visible in the
positive and negative centralizer supports, both within the fundamental
triangle and after unfolding it into the triangular reflection tiling.

More generally, one may study families \(X_n,Y_n\) of axis-symmetric
floretions and classify the set of orders \(n\) for which \(X_nY_n\) has a
triangular reflection symmetry.  The parity-dependent examples above suggest
that such symmetry sets are not governed by ordinary commutation alone, but may
also depend on finer sign constraints, cancellations between signed
centralizer components, and occasional low-order symmetry enhancements.

\section*{Acknowledgements}

The author thanks readers and collaborators who have contributed to the development of the algebraic and geometric viewpoint on floretions, in particular to Neil Sloane for his encouragement of the author's early work.

\clearpage
\appendix

\section{Proof of the parity-dependent product formulas}
\label{app:axis-landing-proof}

We prove the product identities used in Subsection~\ref{subsec:parity-axis-landing}.  The proof is by
induction on the order \(n\).  The main point is that the identities needed
in the paper are not quite closed under induction by themselves.  To make the
induction close, we temporarily keep track of two parity bits.

Throughout this appendix, all parity subscripts are read modulo \(2\); in particular, expressions such as \(\varepsilon+1\), \(\alpha+1\), and \(\alpha+\delta\) are interpreted in \(\mathbb F_2\).  Let
\[
  \iota_n:=\underbrace{ii\cdots i}_{n\ \mathrm{times}}\in\Dn,\qquad
  I_n=\sum_{x\in\{e,i\}^n}x,
\]
and put
\[
  C_n^-=\Sigma_-(\iota_n),\qquad
  C_n^+=\Sigma_+(\iota_n),\qquad
  C_n^b=C_n^-+C_n^+ .
\]

For a single local symbol \(s\in\{e,i,j,k\}\), define two bits
\[
  p(s)=
  \begin{cases}
    0,& s\in\{e,i\},\\
    1,& s\in\{j,k\},
  \end{cases}
  \qquad
  q(s)=
  \begin{cases}
    0,& s\in\{e,k\},\\
    1,& s\in\{i,j\}.
  \end{cases}
\]
Equivalently,
\[
\begin{array}{c|cccc}
s      & e & i & j & k\\ \hline
p(s)  & 0 & 0 & 1 & 1\\
q(s)  & 0 & 1 & 1 & 0
\end{array}
\]
The bit \(p\) records whether \(s\) anticommutes with \(i\): the symbols
\(j\) and \(k\) anticommute with \(i\), while \(e\) and \(i\) commute with
\(i\).  The bit \(q\) records the sign of the local product \(si\), since
\[
  ei=i,\qquad ii=-e,\qquad ji=-k,\qquad ki=j.
\]
Thus \(q(s)=1\) exactly when the product \(si\) contributes a negative local
sign.

For \(\varepsilon,\alpha\in\mathbb F_2\), define the auxiliary word-sum
\[
  \mathcal D_{\varepsilon,\alpha}^{(n)}
  :=
  \sum_{\substack{h=h_1\cdots h_n\in\{e,i,j,k\}^n\\
                  \sum_r p(h_r)\equiv \varepsilon\\
                  \sum_r q(h_r)\equiv \alpha}}
  h.
\]
The centralizer components used in the paper are the \(\varepsilon=0\) cases:
\[
  C_n^+=\mathcal D_{0,0}^{(n)},\qquad
  C_n^-=\mathcal D_{0,1}^{(n)},\qquad
  C_n^b=C_n^++C_n^- .
\]
Indeed, the condition
\[
  \sum_r p(h_r)\equiv0
\]
says that \(h\) has an even number of coordinates that anticommute with
\(i\), and hence that \(h\) centralizes
\(\iota_n=ii\cdots i\).  Once \(h\) centralizes \(\iota_n\), the condition
\[
  \sum_r q(h_r)\equiv\alpha
\]
records whether the common product \(h\iota_n=\iota_nh\) has positive sign
(\(\alpha=0\)) or negative sign (\(\alpha=1\)).

For \(a,b\in\{e,i,j,k\}\) and \(\delta\in\mathbb F_2\), define the parity
refinement
\[
  E_{ab,\delta}^{(n)}
  :=
  \sum_{\substack{w\in\{a,b\}^n\\
                  \#\{r:w_r=b\}\equiv \delta}}
  w.
\]
The notation used in the body of the paper is the even part
\[
  E_{ab}^{(n)}=E_{ab,0}^{(n)}.
\]

We shall prove the stronger identity
\[
  \mathcal D_{\varepsilon,\alpha}^{(n)}I_n
  =
  2^{n-1}
  \left(
    E_{ij,\varepsilon}^{(n)}
    +
    (-1)^\alpha E_{ek,\varepsilon}^{(n)}
  \right)
  \tag{A.1}
\]
for all \(n\geq1\) and all \(\varepsilon,\alpha\in\mathbb F_2\).  The desired
formulas are the \(\varepsilon=0\) cases.

For \(n=1\), the four sums \(\mathcal D_{\varepsilon,\alpha}^{(1)}\) are
\[
  \mathcal D_{0,0}^{(1)}=e,\qquad
  \mathcal D_{0,1}^{(1)}=i,\qquad
  \mathcal D_{1,0}^{(1)}=k,\qquad
  \mathcal D_{1,1}^{(1)}=j.
\]
Also \(I_1=e+i\).  Therefore
\[
  eI_1=e+i,\qquad
  iI_1=i-e,\qquad
  kI_1=k+j,\qquad
  jI_1=j-k.
\]
These are exactly the four cases of (A.1):
\[
\begin{aligned}
  \mathcal D_{0,0}^{(1)}I_1
    &=E_{ij,0}^{(1)}+E_{ek,0}^{(1)},\\
  \mathcal D_{0,1}^{(1)}I_1
    &=E_{ij,0}^{(1)}-E_{ek,0}^{(1)},\\
  \mathcal D_{1,0}^{(1)}I_1
    &=E_{ij,1}^{(1)}+E_{ek,1}^{(1)},\\
  \mathcal D_{1,1}^{(1)}I_1
    &=E_{ij,1}^{(1)}-E_{ek,1}^{(1)}.
\end{aligned}
\]
Thus the induction begins.

We now record the recursive decompositions used in the induction step.  Write
words of length \(n\) as words of length \(n-1\) followed by one last symbol.
Since
\[
p(e)=0,\ q(e)=0;\qquad p(i)=0,\ q(i)=1;
\]
\[
p(j)=1,\ q(j)=1;\qquad p(k)=1,\ q(k)=0,
\]
we have
\[
\begin{aligned}
  \mathcal D_{\varepsilon,\alpha}^{(n)}
  ={}&
  \mathcal D_{\varepsilon,\alpha}^{(n-1)}e
  +\mathcal D_{\varepsilon,\alpha+1}^{(n-1)}i\\
  &+\mathcal D_{\varepsilon+1,\alpha+1}^{(n-1)}j
  +\mathcal D_{\varepsilon+1,\alpha}^{(n-1)}k .
\end{aligned}
\tag{A.2}
\]
For example, the term
\(\mathcal D_{\varepsilon,\alpha+1}^{(n-1)}i\) means: take every word in
\(\mathcal D_{\varepsilon,\alpha+1}^{(n-1)}\) and append the symbol \(i\).
Appending \(i\) changes the \(q\)-parity by \(1\) and leaves the \(p\)-parity
unchanged, so the prefix must have parities \((\varepsilon,\alpha+1)\) in
order for the full word to have parities \((\varepsilon,\alpha)\).

Likewise,
\[
  I_n=I_{n-1}(e+i).
  \tag{A.3}
\]

The \(E\)-sums satisfy analogous even/odd recurrences.  Appending the first
symbol in the pair preserves the number of occurrences of the second symbol,
while appending the second symbol flips its parity.  Hence
\[
  E_{ij,\varepsilon}^{(n)}
  =
  E_{ij,\varepsilon}^{(n-1)}i
  +
  E_{ij,\varepsilon+1}^{(n-1)}j,
  \tag{A.4}
\]
and
\[
  E_{ek,\varepsilon}^{(n)}
  =
  E_{ek,\varepsilon}^{(n-1)}e
  +
  E_{ek,\varepsilon+1}^{(n-1)}k.
  \tag{A.5}
\]
For instance, a word in \(\{i,j\}^n\) has an \(\varepsilon\)-parity number
of \(j\)'s either by appending \(i\) to a previous \(\varepsilon\)-parity
word, or by appending \(j\) to a previous \((\varepsilon+1)\)-parity word.

Assume now that (A.1) holds in order \(n-1\).  We prove it in order \(n\).
The local products with the last-coordinate factor \(e+i\) are
\[
  e(e+i)=e+i,\qquad
  i(e+i)=i-e,
\]
\[
  j(e+i)=j-k,\qquad
  k(e+i)=k+j.
  \tag{A.6}
\]
Using (A.2), (A.3), and the digitwise nature of multiplication, we obtain
\[
\begin{aligned}
  \mathcal D_{\varepsilon,\alpha}^{(n)}I_n
  ={}&
  \mathcal D_{\varepsilon,\alpha}^{(n-1)}I_{n-1}(e+i)
  +\mathcal D_{\varepsilon,\alpha+1}^{(n-1)}I_{n-1}(i-e)\\
  &+\mathcal D_{\varepsilon+1,\alpha+1}^{(n-1)}I_{n-1}(j-k)
  +\mathcal D_{\varepsilon+1,\alpha}^{(n-1)}I_{n-1}(k+j).
\end{aligned}
\tag{A.7}
\]

Put
\[
  U_\delta:=E_{ij,\delta}^{(n-1)},\qquad
  V_\delta:=E_{ek,\delta}^{(n-1)},
\]
and set
\[
  s:=(-1)^\alpha.
\]
By the induction hypothesis,
\[
  \mathcal D_{\varepsilon,\alpha}^{(n-1)}I_{n-1}
  =
  2^{n-2}(U_\varepsilon+sV_\varepsilon),
\]
\[
  \mathcal D_{\varepsilon,\alpha+1}^{(n-1)}I_{n-1}
  =
  2^{n-2}(U_\varepsilon-sV_\varepsilon),
\]
\[
  \mathcal D_{\varepsilon+1,\alpha+1}^{(n-1)}I_{n-1}
  =
  2^{n-2}(U_{\varepsilon+1}-sV_{\varepsilon+1}),
\]
and
\[
  \mathcal D_{\varepsilon+1,\alpha}^{(n-1)}I_{n-1}
  =
  2^{n-2}(U_{\varepsilon+1}+sV_{\varepsilon+1}).
\]
Substituting these four expressions into (A.7) gives
\[
\begin{aligned}
  \mathcal D_{\varepsilon,\alpha}^{(n)}I_n
  =
  2^{n-2}\bigl[
  &(U_\varepsilon+sV_\varepsilon)(e+i)
   +(U_\varepsilon-sV_\varepsilon)(i-e)\\
  &+(U_{\varepsilon+1}-sV_{\varepsilon+1})(j-k)
   +(U_{\varepsilon+1}+sV_{\varepsilon+1})(k+j)
  \bigr].
\end{aligned}
\tag{A.8}
\]

Now collect terms.  The \(U_\varepsilon\)-terms are
\[
  U_\varepsilon(e+i)+U_\varepsilon(i-e)=2U_\varepsilon i.
\]
The \(V_\varepsilon\)-terms are
\[
  sV_\varepsilon(e+i)-sV_\varepsilon(i-e)=2sV_\varepsilon e.
\]
The \(U_{\varepsilon+1}\)-terms are
\[
  U_{\varepsilon+1}(j-k)+U_{\varepsilon+1}(k+j)
  =2U_{\varepsilon+1}j.
\]
The \(V_{\varepsilon+1}\)-terms are
\[
  -sV_{\varepsilon+1}(j-k)+sV_{\varepsilon+1}(k+j)
  =2sV_{\varepsilon+1}k.
\]
Therefore (A.8) becomes
\[
  \mathcal D_{\varepsilon,\alpha}^{(n)}I_n
  =
  2^{n-1}
  \left(
    U_\varepsilon i+U_{\varepsilon+1}j
    +
    s(V_\varepsilon e+V_{\varepsilon+1}k)
  \right).
\]
Using the recurrences (A.4) and (A.5), this is
\[
  \mathcal D_{\varepsilon,\alpha}^{(n)}I_n
  =
  2^{n-1}
  \left(
    E_{ij,\varepsilon}^{(n)}
    +
    (-1)^\alpha E_{ek,\varepsilon}^{(n)}
  \right),
\]
which is (A.1) in order \(n\).  The induction is complete.

Taking \(\varepsilon=0\) in (A.1) gives
\[
  C_n^+I_n
  =
  2^{n-1}
  \left(E_{ij}^{(n)}+E_{ek}^{(n)}\right),
\]
and
\[
  C_n^-I_n
  =
  2^{n-1}
  \left(E_{ij}^{(n)}-E_{ek}^{(n)}\right).
\]
Adding these two identities gives
\[
  C_n^bI_n=2^nE_{ij}^{(n)}.
\]

The reverse-order identities follow from the \(I\)-axis reflection.  This
reflection fixes \(e\) and \(i\), swaps \(j\) and \(k\), fixes \(I_n\), and
fixes the two centralizer components \(C_n^+\) and \(C_n^-\).  Since it is an
odd digit permutation, it reverses multiplication order.  Applying it to the
three identities above yields
\[
  I_nC_n^-
  =
  2^{n-1}
  \left(E_{ik}^{(n)}-E_{ej}^{(n)}\right),
\]
\[
  I_nC_n^+
  =
  2^{n-1}
  \left(E_{ik}^{(n)}+E_{ej}^{(n)}\right),
\]
and
\[
  I_nC_n^b=2^nE_{ik}^{(n)}.
\]

\section{A note on earlier notation and integer sequences}

The order-two case was the first case studied computationally.  In some of the
early OEIS-related work, before the general triangular-coordinate model used in
this paper, order-two basis elements were often written with arrows.  For
example, what is written here as \(ei\) and \(ie\) appeared as \(\li\) and
\(\ri\), respectively, and similarly for \(j\) and \(k\).  This notation was
useful for the original order-two experiments, and it still occurs in a number
of OEIS entries and comments related to floretions.  The word notation used in
the present paper is adopted because it extends uniformly to arbitrary order.

Earlier experiments with order-two floretions produced integer sequences by
inspecting basis coefficients in the powers
\[
  X,\ X^2,\ X^3,\ldots .
\]
The matrix-algebra viewpoint explains why linear recurrence phenomena appear.
After complexification, the order-two algebra is
\[
  A_2\otimes_{\mathbb R}\mathbb C\simeq M_4(\mathbb C).
\]
Thus the Cayley--Hamilton theorem gives a linear recurrence of degree at most
\(4\) for every fixed basis coefficient of \(X^m\).  This accounts for one part
of the older Fibonacci--Lucas--Pell experiments.

The broader subject of floretion-generated integer sequences is larger.  Some
examples arose from coefficient-extraction procedures, from identities involving
positive and negative coefficient parts, from noncommutative word expansions
related to necklace counting, and from musical sequence examples associated
with the OEIS and related projects; see, for example, the ``Sound of Sequences''
presentation of a floretion-generated sequence \cite{SoundOfSequences}.  These
directions are mentioned only for historical context; they are not used in the
coordinate, symmetry, or centralizer results above.

The following small example records a direct exit ramp from the present
coordinate model to second-order linear recurrences.

\begin{exa}[A second-order recurrence inside order two]
Let
\[
  E'=\frac14(ie+ei+ii+jj+kk+jk+kj+ee)
\]
and let
\[
  X=Aei+Bej+Cek .
\]
Put
\[
  Z=E'X .
\]
A direct calculation in the order-two floretion algebra gives
\[
  Z^3+A Z^2+BC\,Z=0 .
\]
Consequently, for every basis word \(u\), if \(a_m\) denotes the coefficient of
\(u\) in \(Z^m\), then for \(m\geq 3\)
\[
  a_m=-A\,a_{m-1}-BC\,a_{m-2}.
\]
For example, choosing \(A=-1\), \(B=1\), and \(C=-1\) gives
\[
  a_m=a_{m-1}+a_{m-2}.
\]
With this choice, the coefficient of \(ij\) in \(Z^m\) is
\[
  \frac12,\ \frac12,\ 1,\ \frac32,\ \frac52,\ 4,\ldots,
\]
that is, one half of the Fibonacci sequence with the usual initial values
\(F_1=F_2=1\).
\end{exa}

\begin{exa}[A Padovan recurrence inside order two]
Let
\[
  E=\frac14(3ee+ii+jj-kk)
\]
and let
\[
  X=\frac12(ie+ij+ik-je+ji+kj).
\]
Put
\[
  Y=EX .
\]
A direct calculation in the order-two floretion algebra gives
\[
  Y^4=Y^2+Y .
\]
Consequently, for every basis word \(u\), if \(a_m\) denotes the coefficient of
\(u\) in \(Y^m\), then for \(m\geq 1\)
\[
  a_{m+3}=a_{m+1}+a_m.
\]
In particular,
\[
  4[ik]Y^m=1,\ 1,\ 1,\ 2,\ 2,\ 3,\ 4,\ 5,\ 7,\ 9,\ 12,\ldots
\]
for \(m=1,2,3,\ldots\), giving the Padovan sequence with initial values
\(1,1,1\).
\end{exa}

These examples are included only as pointers to the recurrence side of the
subject.  They also suggest an inverse recurrence problem.  If a distinguished
coefficient sequence, such as the identity coefficient \([ee\cdots e]X^m\), satisfies a linear
recurrence of prescribed order, what geometric or algebraic restrictions does
this impose on the support tiles and coefficients of \(X\)?

\section*{Statements and Declarations}

\noindent\textbf{Funding.}
The author received no financial support for the research, authorship, or
publication of this article.

\medskip
\noindent\textbf{Conflict of interest.}
The author declares that he has no conflict of interest.

\medskip
\noindent\textbf{Use of generative AI.}
During the preparation and revision of this manuscript, the author used
generative AI tools for language editing, structural suggestions, and
exploratory mathematical discussion, including searching for relevant
literature. All suggestions incorporated into the present manuscript were
critically assessed and mathematically worked through by the author. The
author takes full responsibility for the present content of the manuscript.

\end{document}